\documentclass[a4paper,12pt,reqno]{amsart}

\usepackage[T1]{fontenc}
\usepackage[utf8]{inputenc}
\usepackage[english]{babel}
\usepackage{calc}
\usepackage{hyperref}
\hypersetup{colorlinks,
linkcolor=blue,
filecolor=green,
urlcolor=blue,
citecolor=blue}

\usepackage{amsmath}
\usepackage{amssymb}
\usepackage{mathrsfs}
\usepackage{eucal}

\usepackage{amsthm}
\usepackage{enumitem}  
\setlist{itemsep=\medskipamount,font=\normalfont}

\usepackage[normalem]{ulem}

\DeclareMathAlphabet{\pazocal}{OMS}{zplm}{m}{n}

\usepackage{graphicx}
\usepackage{color}
\usepackage{tikz,pgfplots}  \usetikzlibrary{arrows}
\usepackage{tabularx}

\usepackage{cancel}
\usepackage{soul}

\input xy
\xyoption{all}

\usepackage{multicol}

\usepackage{csquotes}

\newtheorem{theorem}{Theorem}[section]
\newtheorem{lemma}[theorem]{Lemma}
\newtheorem{corollary}[theorem]{Corollary}
\newtheorem{proposition}[theorem]{Proposition}

\theoremstyle{definition}
\newtheorem{definition}[theorem]{Definition}
\newtheorem{example}[theorem]{Example}

\newtheorem{conj}[theorem]{Conjecture}

\newcommand{\comment}[1]{}

\newcommand{\K}{\mathsf{k}}

\def\Z{\mathbb Z}
\def\N{\mathbb N}

\newcommand{\p}{\mathsf{p}}

\newcommand{\gr}{\operatorname{gr}}

\title{Williams' conjecture holds for meteor graphs}

\author[L. G. Cordeiro]{L. G. Cordeiro}
\address{Luiz Gustavo Cordeiro: Departamento de Engenharias da Mobilidade - UFSC - Joinville - SC, Brazil}
\email{luiz.cordeiro@ufsc.br}

\author[E. Gillaspy]{E. Gillaspy}

\address{Elizabeth Gillaspy: 
Mathematical Sciences Department\\
University of Montana\\
USA} \email{elizabeth.gillaspy@mso.umt.edu}

\author[D. Gon\c{c}alves]{D. Gon\c{c}alves}
\address{Daniel Gon\c{c}alves: Departamento de Matem\'{a}tica - UFSC - Florian\'{o}polis - SC, Brazil}
\email{daemig@gmail.com}

\author[R. Hazrat]{R. Hazrat}

\address{Roozbeh Hazrat: 
Centre for Research in Mathematics and Data Science\\
Western Sydney University\\
Australia} \email{r.hazrat@westernsydney.edu.au}



\begin{document}
\maketitle

\begin{abstract}
      A meteor graph is a connected graph with no sources and sinks consisting of two disjoint cycles and the paths connecting these cycles.
    We prove that two meteor graphs are shift equivalent if and only if they are strongly shift equivalent, if and only if their corresponding Leavitt path algebras are graded Morita equivalent,  if and only if their graded $K$-theories, $K_0^{\gr}$, are $\mathbb Z [x,x^{-1}]$-module isomorphic. As a consequence, the Leavitt path algebras of meteor graphs are graded Morita equivalent if and only if their graph $C^*$-algebras are equivariant Morita equivalent. 
    
\end{abstract}

\section{Introduction}

Shifts of finite type are central objects in the theory of symbolic dynamics; 
an isomorphism between two shifts of finite type is called a {\em conjugacy}. 
Up to conjugacy, every shift of finite type arises from an {\em essential graph} -- that is, a finite connected  directed graph $E$ with no sinks or sources \cite{lindmarcus}. The shift space $X_E$ associated to the  graph $E$ is given by the set of bi-infinite paths in $E$, and the natural shift of the paths to the left. This is called an \emph{edge shift}. 

Determining whether two shifts of finite type $X_E, X_F$ are conjugate is in general a difficult problem, because it requires knowledge of all possible bi-infinite paths in $E$ and $F$. In his seminal paper~\cite{williams}, Williams introduced the notions  of {\em shift equivalence} (SE) and {\em strong shift equivalence} (SSE), which are more tractable.  Williams showed in \cite{williams} that strong shift equivalence completely characterizes conjugacy, and described SSE as the equivalence relation  generated by the graph moves of in/outsplitting and their inverses.  (We describe these moves in detail in Section \ref{dynref} below.)  


 Shift equivalence is a weaker equivalence relation than SSE, and is also more computable, as Krieger established in \cite{krieger}. Williams originally asserted in \cite{williams} that SE $\iff$ SSE.  Although he identified a flaw in the proof that SE $\implies$ SSE  a year later \cite{willwrong}, it took 25 years before a counterexample to SE $\implies$ SSE was found, by Kim and Roush in \cite{kimroush}.  

Identifying classes of shifts for which SE and SSE are equivalent is an open problem. Even in the case of $2\times 2$ matrices, examples are abound~\cite[Example 7.3.13]{lindmarcus}: the matrices 
\[A_k=\begin{pmatrix}
1 & k \\
k-1 &1\\
\end{pmatrix} \text{   and   } 
B_k=\begin{pmatrix}
1 & k(k-1)\\
1 & 1\\
\end{pmatrix},\] are shift equivalent for $k\geq 1$  
but are only known to be strongly shift equivalent for  $1 \leq k\leq 3$.

The notions of SSE and SE have very interesting interpretations in the theory of operator algebras, and extensive research related to these notions is ongoing (see~\cite{toke, toke2} and the references within).

In this note, we prove that the notions of SSE and SE coincide for the class of meteor graphs. A \emph{meteor graph} is a connected essential graph consisting of two disjoint cycles {and the paths connecting these cycles}; see Definition \ref{def:meteor} below. 
Several tools we use are well  known, namely, Williams' graph moves and Krieger's dimension theory (\S\ref{dynref}). Our key idea is to make a bridge from symbolic dynamics, via the theory of Leavitt path algebras, to the notion of \emph{talented monoids}~\cite{hazli} as follows: If essential graphs $E$ and $F$ are shift equivalent, then their Krieger's dimension groups are isomorphic, $\Delta(E) \cong \Delta(F)$. It is known that $\Delta(E)\cong K_0^{\gr}(L(E))$, the latter being the graded Grothendieck group of the Leavitt path algebra $L(E)$~\cite{arapar, hazd}. Thus we obtain a $\mathbb Z[x,x^{-1}]$-order isomorphism $K_0^{\gr}(L(E)) \cong K_0^{\gr}(L(F))$. The positive cone of the graded Grothendieck group can be described purely based on the underlying graph, via the so-called talented monoid $T_E$ of $E$. Therefore, from a shift equivalence of graphs we obtain a $\mathbb Z$-monoid isomorphism $T_E\cong T_F$. This isomorphism gives us control over elements of the monoids (such as minimal elements, atoms, etc.) and consequently on the geometry of the graphs (the number of cycles, their lengths, etc.). Thus, a careful analysis of the monoid isomorphism allows us to show that $E$ and $F$ can be transformed, via in- and out-splitting and their inverses, 
into each other. Applying Williams' theorem, we obtain that the graphs $E$ and $F$ are strongly shift equivalent and, as a consequence, $L(E)$ is graded Morita equivalent to $L(F)$. This gives the Williams conjecture (SE $\Leftrightarrow$ SSE) and the graded classification conjecture ($K_0^{\gr}(L(E)) \cong K_0^{\gr}(L(F))$ $\Leftrightarrow$ $L(E)$ is graded Morita equivalent to $L(F)$) for meteor graphs. 

The paper is organized as follows:   Section~\ref{sec2} recalls the fundamental concepts we will rely on in this paper, such as $\Gamma$-monoids, the monoids $M_E, T_E$ associated with a directed graph $E$,  in-/out-splitting, and Krieger's dimension group.  To enable our careful analysis of the talented monoid $T_E$, which is a $\Z$-monoid, Section \ref{felhmm} introduces the concept of the $\Gamma$-Archimedean classes of a $\Gamma$-monoid.  The talented monoids of meteor graphs have a particularly nice structure of $\Z$-Archimedean classes, as we establish in Theorem \ref{ghostagain}, and we rely on this to prove our main results in Section \ref{meteorchap}. 
Using the talented monoid associated with a graph, and the tools developed earlier in the paper, we establish Theorem \ref{thm:main} by using Williams' in- and out-splitting to show that the class of meteor graphs is closed under shift equivalence and that SSE and SE coincide for this class of graphs. 

\section{Background: Monoids, Graphs and Algebras}\label{sec2}

\subsection{Monoids, \texorpdfstring{\(\mathbb{Z}\)}{ℤ}-monoids and order ideals}

A \emph{semigroup} is a set with an associative binary operation. Subsemigroups of a semigroup are defined in the usual sense: these are the nonempty subsets closed under the operation.

A \emph{monoid} is a semigroup whose operation has an identity element. A {\em submonoid} of a monoid is a subsemigroup that also contains the identity. Throughout this paper, we are most interested in \emph{commutative monoids}, which are those with a commutative operation. In this case, the operation is written additively (with the symbol \(+\)) and the unit is denoted as zero (\(0\)). A commutative monoid \(M\) is called \emph{conical} if \(x+y=0\) in \(M\) implies \(x=y=0\). The monoid $M$ is called \emph{cancellative} if \(x+a=y+a\) in \(M\) implies \(x=y\).

Given a commutative monoid \(M\), we define the \emph{algebraic preorder} on \(M\) by setting \(x\leq y\) if and only if \(y=x+z\) for some \(z\in M\). If \(M\) is conical and cancellative, then \(\leq\) is a partial order. Conversely, if \(\leq\) is a partial order then \(M\) is conical (but not necessarily cancellative).

A \emph{$\Gamma$-monoid}, where $\Gamma$ is an abelian group,  consists of a monoid \(M\) equipped with a group action of 
$\Gamma$ by monoid homomorphisms. The image of an element \(x\in M\) under the action of a group element $\gamma \in \Gamma$ is denoted by \({}^\gamma x\). Throughout this paper, we work with $\mathbb Z$-monoids. 


The set of natural numbers is denoted by \(\mathbb{N}=\left\{0,1,2,\ldots\right\}\). Under the usual sum, it is the free monoid generated by a single element.

One of the $\mathbb Z$-monoids we encounter in this paper is the following. 

\begin{definition}\label{cycmond}
    Let \(k\) be a positive integer. The monoid $T=\bigoplus_{i=1}^k \mathbb N$, with the action of $\mathbb Z$ defined by ${}^1(a_1,\dots,a_{k-1},a_k)=(a_k,a_1\dots, a_{k-1})$, is called the \emph{$\mathbb Z$-cyclic monoid of rank \(k\)}.
\end{definition}

\subsection{Graphs and associated monoids}

A \emph{directed graph} \(E\) (which we will often refer to simply as a \emph{graph}) consists of two sets, \(E^0\) and \(E^1\), whose elements are called \emph{vertices} and \emph{edges}, respectively, and two functions \(s,r\colon E^1\to E^0\), called the \emph{source} and \emph{range} maps, respectively. In this paper, we are interested in \emph{row-finite} graphs, that is, graphs such that $|s^{-1}(v)|<\infty$ for all $v\in E^0$.

A {\em sink} in a directed graph $E$ is a vertex $v \in E^0$ with $|s^{-1}(v)| = 0;$ a {\em source} is a vertex $v \in E^0$ with $|r^{-1} v| =0.$  A directed graph $E$ is {\em essential} if it has no sinks or sources.  We will focus our attention on essential graphs in this paper.

Let $E$ be a directed graph. A (finite) \emph{path} in \(E\) is a string
$\alpha=\alpha_1\cdots\alpha_n$
of edges \(\alpha_i\in E^1\) which satisfy $r(\alpha_i) = s(\alpha_{i+1})$ for all $i$. The \emph{length} of the path \(\alpha = \alpha_1 \cdots \alpha_n \) is $n$, and is denoted $|\alpha|$. 
The source and range maps on edges are extended to paths as
\[s(\alpha_1\cdots\alpha_n)=s(\alpha_1)\qquad\text{and}\qquad r(\alpha_1\cdots\alpha_n)=r(\alpha_n).\]
Vertices are  regarded as paths of length \(0\), with each vertex coinciding with its source and its range.



A vertex \(v\) is said to \emph{lie} on a path \(\alpha\) if \(v\) is the source or the range of one of the edges which comprise \(\alpha\). The set of vertices that lie on \(\alpha\) is denoted by \(\alpha^0\). 

A path \(\alpha=\alpha_1\cdots\alpha_n\) is \emph{simple} if the restrictions of 
$s$ and $r$ to \(\left\{\alpha_1,\cdots,\alpha_n\right\}\) are injective.
A \emph{cycle} in \(E\) is a simple path \(\alpha\) of length at least \(1\) 
for which \(s(\alpha)=r(\alpha)\).

\begin{definition}\label{def:graphmonoid}
    Let $E$ be a row-finite graph. We define $F_E$ to be the free commutative monoid generated by $E^0$. The \emph{graph monoid} of $E$, denoted $M_E$, is the 
    quotient of $F_E$ by the relation 
\[v=\sum_{e\in s^{-1}(v)}r(e)\]
for every $v\in E^0$ that is not a sink.
\end{definition}

The relations defining $M_E$ can be described more concretely as follows: First, define a relation $\to_1$ on $F_E$ as follows: for $\sum_{i=1}^n v_i  \in F_E$,  
and any $1\leq j \leq n$, set 
\begin{equation}\label{arrow11}
\sum_{i=1}^n v_i \to_1 \sum_{i\not = j }^n v_i+  \sum_{e\in s^{-1}(v_j)}r(e).
\end{equation}

Then $M_E$ is the quotient of $F_E$ by the {congruence} $\to$ generated by $\to_1$.  To be precise, 
 $\to$ is the smallest reflexive, transitive and additive relation on $F_E$ which contains 
 $\to_1$. This relation may be regarded as follows: If \(x=\sum_i x_i\) is an element of \(F_E\), we may ``let a vertex \(x_i\) flow'' to construct the element \(y_1=\left(\sum_{j\neq i}x_j\right)+\sum_{e\in s^{-1}(x_i)}r(e)\) with \(x\to y_1\). Repeating this procedure and ``letting a vertex of \(y_1\) flow'', we construct another element \(y_2\in F_E\) such that \(y_1\to y_2\). In other words, we simply apply the definition of \(\to_1\) to vertices in the representation of elements of \(F_E\). By the definition of \(\to\), every element \(y\in F_E\) such that \(x\to y\) may be constructed from \(x\) by ``letting its vertices flow successively'' in this manner. The following proposition thus becomes clear:

\begin{proposition}\label{prop:flowpath}
Suppose that $x=\sum_i x_i$ and $y=\sum_j y_j$ are elements of $F_E$, where $x_i,y_j\in E^0$. If $x\to y$, then
\begin{enumerate}[label=(\arabic*)]
    \item For every $i$, there exists $j$ such that there is a path in $E$ from $x_i$ to $y_j$;

    \item For every $j$, there exists $i$ such that there is a path in $E$ from $x_i$ to $y_j$.
\end{enumerate}
\end{proposition}

 By the proposition above, a vertex $v$ flows to the vertex $u$ if, and only if, either $v=u$, or there exists $x\in F_E$ such that $u$ 
 is a summand of $x$, and such that $v\to x$.

The following lemma is essential to the remainder of this paper, as it allows us to translate the relations in the definition of $M_E$ in terms of the simpler relation $\to$ in $F_E$.

\begin{lemma}[{\cite[Lemmas 4.2 and 4.3]{amp}}]\label{confuu}
    Let $E$ be a row-finite graph.
    \begin{enumerate}[label=(\arabic*)]
        \item (The Confluence Lemma) If $a,b\in F_E\setminus\left\{0\right\}$, then $a=b$ in $M_E$ if and only if there exists $c\in F_E$ such that $a\to c$ and $b\to c$. (Note that, in this case, $a=b=c$ in $M_E$.)

        \item If $a=a_1+a_2$ and $a\to b$ in $F_E$, then there exist $b_1,b_2\in F_E$ such that $b=b_1+b_2$, $a_1\to b_1$ and $a_2\to b_2$.
    \end{enumerate}
\end{lemma}

Next we define the \emph{talented monoid} $T_E$ of $E$, which is believed to encode the graded structure of the Leavitt path algebra $L_{\mathsf k}(E)$ (see Conjecture~\ref{conjehfyhtr}) and, later in the paper, plays the role of a bridge between symbolic dynamics and the theory of Leavitt path algebras.

\begin{definition}\label{talentedmon}
Let $E$ be a row-finite directed graph. The \emph{talented monoid} of $E$, denoted $T_E$, is the commutative 
monoid generated by $\{v(i) \mid v\in E^0, i\in \mathbb Z\}$, subject to
\[v(i)=\sum_{e\in s^{-1}(v)}r(e)(i+1)\]
for every $i \in \mathbb Z$ and every $v\in E^{0}$ that is not a sink. The additive group $\mathbb{Z}$ of integers acts on $T_E$ via monoid automorphisms by shifting indices: For each $n,i\in\mathbb{Z}$ and $v\in E^0$, define ${}^n v(i)=v(i+n)$, which extends to an action of $\mathbb{Z}$ on $T_E$. Throughout the paper we denote the elements $v(0)$ in $T_E$ by $v$.
\end{definition}

The talented monoid of a graph can also be seen as the graph monoid of the so-called \emph{covering graph} of $E$, denoted by $\overline{E}$: 
we have   $\overline{E}^0=E^0\times\mathbb{Z}$,  $\overline{E}^1=E^1\times\mathbb{Z}$ and the range and source maps are given by
\[s(e,i)=(s(e),i),\qquad r(e,i)=(r(e),i+1).\] Note that the graph monoid $M_{\overline E}$ has a natural $\mathbb Z$-action by ${}^n (v,i)= (v,i+n)$. The following theorem allows us to use the Confluence Lemma~\ref{confuu} for the talented monoid $T_E$ by identifying it with $M_{\overline E}$. 

\begin{theorem}[{\cite[Lemma 3.2]{hazli}}]\label{hgfgfgggf}
    The correspondence
    \begin{align*}
        T_{E} &\longrightarrow M_{\overline{E}}\\
        v(i) &\longmapsto (v,i)
    \end{align*}
    induces a $\mathbb Z$-monoid isomorphism.
\end{theorem}

Lemma~5.5 of \cite{arahazrat} establishes that since $\overline{E}$ has no directed cycles, $M_{\overline{E}} = T_E$ is a cancellative monoid for all $E.$

\subsection{Leavitt path algebras}\label{leviig}

Let $E=(E^{0}, E^{1}, r, s)$ be a directed graph.  One can associate an algebra with coefficients in a field $\K$, $L_\K(E)$, to the graph $E$, which is called the \emph{Leavitt path algebra}. We refer the reader to the book of Abrams, Ara, and Siles Molina \cite{lpabook} for the theory of Leavitt path algebras, its relation with the algebras defined by William Leavitt, and all the standard terminologies we use here. Throughout the paper, we simply write $L(E)$ for the Leavitt path algebra associated to a directed graph $E$ with coefficients from a field $\K$. 

Finding a complete invariant for the classification of Leavitt path algebras is an ongoing endeavor. 
The Graded Classification Conjecture~(\cite{mathann,hazd}, \cite[\S 7.3.4]{lpabook})  roughly predicts that the graded Grothendieck group $K_0^{\gr}$ classifies
Leavitt path algebras of finite graphs, up to graded isomorphism. The conjecture is closely related to Williams' conjecture (see~\S\ref{dynref}).

In order to state the Graded Classification Conjecture, we first recall the definition of the graded Grothendieck group of a $\Gamma$\!-graded ring. 
Given a $\Gamma$\!-graded ring $A$ with identity and a graded finitely generated projective (right) $A$-module $P$, let $[P]$ denote the class of graded $A$-modules graded isomorphic to $P$. Then the monoid  
\begin{equation}\label{zhongshan1}
\mathcal V^{\gr}(A)=\big \{ \, [P] \mid  P  \text{ is a graded finitely generated projective A-module} \, \big \}
\end{equation}
has a $\Gamma$\!-module structure defined as follows: for $\gamma \in \Gamma$ and $[P]\in \mathcal V^{\gr}(A)$, $\gamma .[P]=[P(\gamma)]$. 

The group completion of $\mathcal V^{\gr}(A)$ is called the \emph{graded Grothendieck group} and is denoted by $K^{\gr}_0(A)$.  The $\Gamma$-\!module structure on $\mathcal V^{\gr}(A)$ induces 
a $\Z[\Gamma]$-module structure on the group $K^{\gr}_0(A)$. In particular, the graded Grothendieck group of a $\mathbb Z$-graded ring has a natural $\mathbb Z[x,x^{-1}]$-module structure. 

\begin{conj}\label{conjehfyhtr}[{\sc The Graded Classification Conjecture}]
Let $E$ and $F$ be finite graphs. Then the following are equivalent.

\begin{enumerate}[label=(\arabic*)]
        \item The Leavitt path algebras \(L(E)\) and \(L(F)\) are graded Morita equivalent,

        \item The $C^*$-algebras $C^*(E)$ and $C^*(F)$ are equivariant Morita equivalent,

        \item There is an order-preserving 
        $\mathbb Z[x,x^{-1}]$-module isomorphism
$K_0^{\gr}(L(E))\rightarrow K_0^{\gr}(L(F))$,

        \item The talented monoids \(T_E\) and \(T_F\) are \(\mathbb{Z}\)-isomorphic.

    \end{enumerate}

\end{conj}

Furthermore, if the isomorphism between the groups in (4) or the monoids in (3) is pointed, i.e., preserves the order-unit, then the algebras should be (graded/gauge invariant) isomorphic (see \cite{mathann}).  

We refer the reader to \cite{arapar, guido,guido1,guidowillie, toke,eilers2, vas} for works on the graded classification conjecture and~\cite{comb, ef} for equivariant Morita equivalence of $C^*$-algebras.

\subsection{Symbolic Dynamics}\label{dynref}

The notion of shift equivalence for matrices was introduced by Williams in \cite{williams} (see also~\cite[\S7]{lindmarcus}) in an attempt to provide computable machinery for determining the conjugacy between two shifts of finite type. Recall that two square nonnegative integer matrices $A$ and $B$ are called {\it elementary shift equivalent}, and denoted by $A\sim_{ES} B$, if there are nonnegative matrices $R$ and $S$ such that $A=RS$ and $B=SR$. 
The equivalence relation $\sim_S$  on square nonnegative integer matrices generated by elementary shift equivalence is called {\it strong shift equivalence}. The weaker notion of shift equivalence is defined as follows. The nonnegative integer matrices $A$ and $B$ are called {\it shift equivalent} if there are nonnegative matrices $R$ and $S$ such that $A^l=RS$ and $B^l=SR$, for some $l\in \mathbb N$, and 
$AR=RB$ and $SA=BS$. 

The {\em adjacency matrix} $A_E\in M_{E^0}(\Z_{\geq 0})$ of a directed graph $E$ provides the link between symbolic dynamics and directed graphs.  By definition, $A_E$ is a square matrix with \[ A_E(v,w) = \left| \left( s^{-1}(v) \cap r^{-1}(w)\right)\right| .\]  Conversely,  any $A \in M_n(\Z_{\geq 0})$ can be interpreted as the adjacency matrix on a graph $E$ with $|E^0| = n$;  $E^1$ consists of precisely $A(i,j)$ edges from vertex $i$ to vertex $j$, for all $1 \leq i, j \leq n$.

Identifying a square nonnegative integer matrix with its associated graph (and the graph with its adjacency matrix), Williams showed that two matrices $A$ and $B$ are strongly shift equivalent if and only if one can reach from $A$ to be $B$ with certain graph moves. We recall these in and out-splitting graph moves here, as we will employ them throughout the text. 

\subsection*{Move (I): In-splitting}

\begin{definition}\cite[Definition 6.3.20]{Luiz}\label{def:insplit}
    Let $E$ be a directed graph. For each $v\in E^0$ with $r^{-1}(v)\neq \emptyset$, take a partition $\left\{\mathscr{E}^v_1,\ldots,\mathscr{E}^v_{m(v)}\right\}$ of $r^{-1}(v)$. We form a new graph $F$ as follows.  Set
    \[F^0=\left\{v_i \mid v\in E^0,1\leq i\leq m(v)\right\}\cup\left\{v \mid r^{-1}(v)=\emptyset \right\}\]
    \[F^1=\left\{e_j \mid e\in E^1,1\leq j\leq m(s(e))\right\}\cup\left\{e \mid r^{-1}(s(e))=\emptyset \right\},\]
    and define the source and range maps  as follows: If $r^{-1}(s(e))\neq\emptyset $, and 
    $e\in\mathscr{E}^{r(e)}_i$, then
    \[s(e_j)=s(e)_j,\qquad r(e_j)=r(e)_i, \text{ for all } 1\leq j \leq m(s(e)).\]
    If $r^{-1}(s(e))=\emptyset$, set $s(e)$ as the original source of $e$, and $r(e)=r(e)_i$, where 
    $e\in\mathscr{E}^{r(e)}_i$.
    
    The graph $F$ is called an \emph{in-split} of $E$, and conversely $E$ is called an \emph{in-amalgam} of $F$. We say that $F$ is formed by performing \emph{Move (I)} on $E$.
\end{definition}

\subsection*{Move (O): Out-splitting}

The notions dual to those of in-split and in-amalgam are called \emph{out-split} and \emph{out-amalgam}. Given a graph $E=(E^0,E^1,s,r)$, the \emph{transpose graph} is defined as $E^*=(E^0,E^1,r,s)$.

\begin{definition}\cite[Definition 6.3.23]{lpabook}\label{def:outsplit}
    A graph $F$ is an \emph{out-split} (\emph{out-amalgam}) of a graph $E$ if $F^*$ is an in-split (in-amalgam) of $E^*$, and we say that $F$ is formed by performing \emph{Move (O)} on $E$.
    
\end{definition}

We emphasize that in this paper, we require the sets $\mathscr E^i_v$ used in in- or out-splitting to be non-empty.


Now that we have defined in- and out-splitting, we can give the precise form of Williams' result.
Recall \cite[Definition 2.2.5]{lindmarcus} that the shift space $X_E$ associated to a directed graph $E$ (equivalently, the shift space $X_{A_E}$ associated to the adjacency matrix $A_E$ of $E$) consists of the set of infinite paths $\alpha = (\alpha_n)_{n\in \Z}$ in $E$, equipped with the natural left shift.

\begin{theorem}[{\sc Williams}~\cite{williams}]\label{willimove}
Let $A$ and $B$ be two square nonnegative integer matrices and let $E$ and $F$ be two essential graphs.

\begin{enumerate}[label=(\arabic*)]

\item  $X_A$ is conjugate to $X_B$ if and only if $A$ is strongly shift equivalent to $B$. 

\item  $X_E$ is conjugate to $X_F$ if and only if and  $E$ can be obtained from $F$ by a sequence of in/out-splittings and their inverses. 

\end{enumerate}

\end{theorem}

For a proof of this theorem, see \cite[Chapter 7]{lindmarcus}.

Although strong shift equivalence characterizes conjugacy of shifts of finite type, there is no general algorithm for deciding whether $X_A$ and $X_B$ are strongly shift equivalent, even for $2 \times 2$ matrices $A,B.$ The weaker notion of shift equivalence  is easier to analyze, as Krieger established.  In~\cite{krieger}, he 
proposed an invariant for classifying the irreducible shifts of finite type up to shift equivalence. Surprisingly, Krieger's dimension group in symbolic dynamics turns out to be expressible as the graded Grothendieck group of a  Leavitt path algebra. As this fact will be key to our solution to the Williams problem for meteor graphs, we pause to recall the details of this connection.

Let $A$ be a nonnegative integral $n\times n$-matrix. Consider the sequence of free ordered abelian groups $\mathbb Z^n \to \Z^n \to \cdots$, where the ordering in $\mathbb Z^n$ is defined point-wise (i.e.,  the positive cone is $\mathbb N^n$).  Then  $A$ acts as an order-preserving group homomorphism on this sequence, as follows: 
\[\mathbb Z^n \stackrel{A}{\longrightarrow} \mathbb Z^n \stackrel{A}{\longrightarrow}  \mathbb Z^n \stackrel{A}{\longrightarrow} \cdots.
\]
 The direct limit of this system, $\Delta_A:= \varinjlim_{A} \mathbb Z^n$, along with its positive cone, $\Delta^+$, and the automorphism which is induced by $A$ on the direct limit, 
$\delta_A:\Delta_A \rightarrow \Delta_A$, is the invariant considered by Krieger, now known as  \emph{Krieger's dimension group}.  Following~\cite{lindmarcus}, we denote this triple by $(\Delta_A, \Delta_A^+,  \delta_A)$. 

The following theorem was proved by Krieger (\cite[Theorem~4.2]{krieger}, 
see also~\cite[\S7.5]{lindmarcus} for a detailed algebraic treatment). 

\begin{theorem}\label{kriegeror}
Let $A$ and $B$ be two square nonnegative integer matrices. Then, $A$ and $B$ are shift equivalent if and only if 
\[(\Delta_A, \Delta_A^+, \delta_A) \cong (\Delta_B, \Delta_B^+, \delta_B).\]
\end{theorem}

 Wagoner noted that the induced structure on $\Delta_A$ by the automorphism $\delta_A$ makes $\Delta_A$ a $\mathbb Z[x,x^{-1}]$-module.  This fact  was systematically used in~\cite{wago1,wago2} (see also~\cite[\S3]{boyle}).

Let $E$ be a graph and $L(E)$ the associated Leavitt path algebra. Recall that the graded Grothendieck group $K_0^{\gr}(L(E))$  has a natural $\mathbb Z[x,x^{-1}]$-module structure.  
Theorem~\ref{mmpags} shows that the graded Grothendieck group of the Leavitt path algebra  coincides with the Krieger dimension group of the shift of finite type associated with $A_E^t$, the transpose of the adjacency matrix $A_E$: 
\[\big(K_0^{\gr}(L(E)),(K_0^{\gr}(L(E))^+\big ) \cong (\Delta_{A^t}, \Delta_{A^t}^+).\] 
This will provide a link between the theory of Leavitt path algebras and symbolic dynamics~\cite{arapar,hazd,hazbk}.

\begin{theorem}\label{mmpags}
Let $E$ be a finite graph with no sinks with adjacency matrix $A$. Then, there is a $\mathbb Z[x,x^{-1}]$-module isomorphism 
$\phi:K^{\gr}_0(L(E)) \longrightarrow \Delta_{A^t}$ such that $\phi(K^{\gr}_0(L(E))^+)=\Delta_{A^t}^+$. 
\end{theorem}

It is easy to see that two matrices $A$ and $B$ are shift equivalent if and only if $A^t$ and $B^t$ are shift equivalent. Combining this with Theorem~\ref{mmpags} and the fact that Krieger's dimension group is a complete invariant for shift equivalence, we have the following corollary, which will be used in our main Theorem~\ref{alergy1}.

\begin{corollary}\label{h99}
Let $E$ and $F$ be finite graphs with no sinks and $A_E$ and $A_F$ be their adjacency matrices, respectively. Then
$A_E$ is shift equivalent to $A_F$ if and only if $K_0^{\gr}(L(E)) \cong K_0^{\gr}(L(F))$ via an order-preserving $\mathbb Z[x,x^{-1}]$-module isomorphism.
\end{corollary}

\section{Archimedean classes}\label{felhmm}

The goal of this section is Theorem \ref{ghostagain}, which describes the structure of the Archimedean classes of the talented monoid of a meteor graph.  This structure is a key ingredient in our proof, in Theorem \ref{thm:main}, that for meteor graphs, shift equivalence implies strong shift equivalence.

We define meteor graphs in Definition \ref{def:meteor}, in Section \ref{sec:arch-for-meteor}; Section \ref{sec:arch-general} studies Archimedean classes in more generality.

\subsection{Archimedean classes of a talented monoid}
\label{sec:arch-general}
Let \(M\) be a $\Gamma$-monoid, where $\Gamma$ is an abelian group. The 
$\Gamma$-order ideal generated by an element \(x\in M\) is the smallest 
$\Gamma$-order ideal \(\langle x\rangle\) which contains \(x\). Explicitly, 
\begin{equation}
    \langle x\rangle = \big \{ y \in M \mid y\leq \sum_i {}^{\gamma_i}x, \gamma_i \in \Gamma \big \}.
    \label{eq:order-ideal}
\end{equation}

The $\Gamma$-\emph{Archimedean equivalence relation} 
$\asymp$ on $M$ is defined as $x\asymp y$ if 
$\langle x\rangle = \langle y\rangle.$
 More explicitly, \(x\asymp y\) if and only if there exist \(\gamma_1,\ldots,\gamma_k\) and \(\alpha_1,\ldots,\alpha_{k'}\) in $\Gamma$ such that
\[x\leq\sum_{i=1}^k{}^{\gamma_i}y\quad \text{and}\quad y \leq \sum_{i=1}^{k'}{}^{\alpha_i}x.\]

An \(\asymp\)-equivalence class is called an \emph{Archimedean class} of \(M\). We denote the $\asymp$-equivalence class of $x$ by $[x]$. It is easy to see that $[x]$ is a $\Gamma$-semigroup. 

Note that if $\Gamma$ is trivial, we arrive at the standard definition of Archimedean classes of a 
monoid~\cite[\S13]{RGS}. 

\begin{example}
Let $M=\mathbb N \oplus \mathbb N \oplus \mathbb N$ be a monoid with usual component-wise addition. Observe that the Archimedean components of $M$ consist of $8$  semigroups corresponding to the origin, each axis, and each plane of the three-dimensional space. 

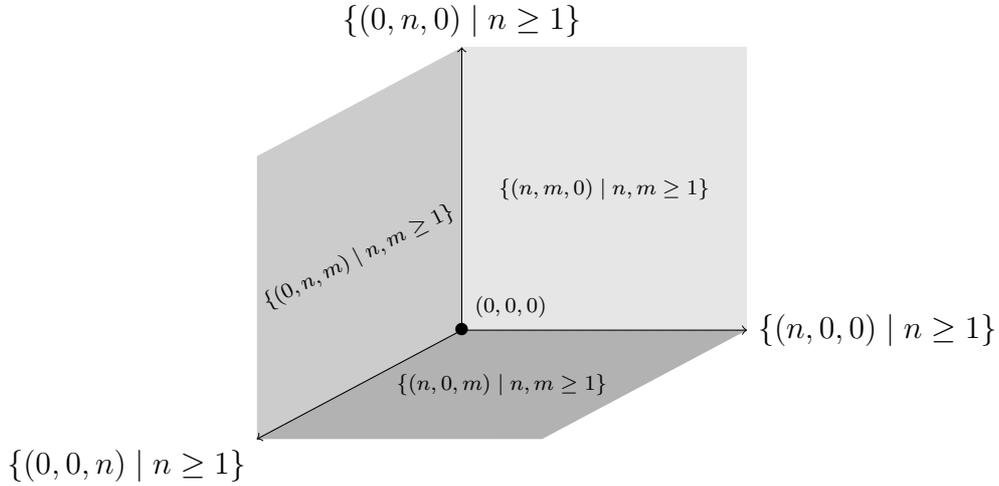
\begin{figure}[!ht]
    \centering
    \begin{tikzpicture}[scale=2.5]
  \coordinate (O) at (0,0,0);
  \coordinate (X) at (1.5,0,0);
  \coordinate (Y) at (0,1.5,0);
  \coordinate (Z) at (-0.5,0,1.5);

  \foreach \c in {X,Y,Z}{
    \draw[dashed] (O) -- (\c);
  }
  
  \fill[black!10] (O) -- (X) -- (X|-Y) -- (Y) -- cycle;
  \fill[black!20] (O) -- (Y) -- (-0.5,1.5,1.5) -- (Z) -- cycle;
  \fill[black!30] (O) -- (X) -- (1,0,1.5) -- (Z) -- cycle;

  \node [label={[label distance=-0.3cm]45:\tiny$(0,0,0)$}] at (O) {$\bullet$};
  \draw[->] (O) -- (X) node[right] {$\left\{(n,0,0)\mid n\geq 1\right\}$};
  \draw[->] (O) -- (Y) node[above] {$\left\{(0,n,0)\mid n\geq 1\right\}$};
  \draw[->] (O) -- (Z) node[below left] {$\left\{(0,0,n)\mid n\geq 1\right\}$};
  \node at (0.5,0,0.75) {\tiny$\left\{(n,0,m)\mid n,m\geq 1\right\}$};
  \node at (0.75,0.75,0) {\tiny$\left\{(n,m,0)\mid n,m\geq 1\right\}$};
  \node[rotate=26.565] at (-0.06,0.87,1.25) {\tiny$\left\{(0,n,m)\mid n,m\geq 1\right\}$};
\end{tikzpicture}
    \caption{Archimedean classes of  the monoid $\mathbb{N}\oplus\mathbb{N}\oplus\mathbb{N}$. The remaining class, $\left\{(m,n,p)\mid m,n,p\geq 1\right\}$, is not depicted.}
    \label{fig:archimedean_n3}
\end{figure}

However, considering $M$ as a \emph{$\mathbb Z$-cyclic monoid of rank \(3\)} (see Definition~\ref{cycmond}), the Archimedean components are only $0$ and $M\backslash 0$. 
    
\end{example}

Let \(E\) be a finite graph. Our immediate goal is to describe the $\mathbb Z$-Archimedean classes of the monoids associated with \(E\) in geometric terms. For this, we start with a construction on subsets of \(E\) which corresponds to the relation \(\to_1\) defined on $F_E$ in~(\ref{arrow11}).

Given a subset \(A\subseteq E^0\), we let
\[A^{1\rightarrow}=r(s^{-1}(A))\cup\left\{a\in A\mid a\text{ is a sink}\right\}.\]

Throughout the paper, we also denote $A^{1\rightarrow}$ by  $A^{\rightarrow}$. For $n \geq 2$, we define \(A^{n\rightarrow}=(A^{(n-1)\rightarrow}){^{1\rightarrow}}\). 
Finally, define
\[R(A)=\bigcap_{n\geq 1}\left(\bigcup_{m\geq n}A^{m\rightarrow}\right).\]
In other words, \(R(A)\) consists of all sinks which can be {flowed} into by some point of \(A\), and all points which can be flowed into from points of \(A\) by paths of arbitrarily large length.

\begin{definition}
    We call the set \(R(A)\) as above the set of \emph{leaves} of \(A \subseteq E^0\). We denote by  \(\mathscr{L}(E)\) the collection of sets of leaves of $E$, that is,  
    \[
    \mathscr{L}(E):=\{R(A)\mid A\subseteq E^0\}.\]
\end{definition}

\begin{proposition}\label{helpprop}
    If $v \in A$, and there is a path $\gamma$ in $E$ from $v$ to a vertex $w$ which lies on a cycle $C$, then every vertex on $C$ lies in $R(A)$.
\end{proposition}
\begin{proof}
    Notice that for any $m \geq 0$, $w \in A^{(m\cdot |C|+ |\gamma| )\to} $.  Therefore $w \in R(A).$  Similarly, if $z$ is another vertex on the cycle $C$, and the path from $w$ to $z$ along the cycle has length $j$, then $z \in A^{(m\cdot | C |+ j + | \gamma|)\to}$ for all $m \geq 0$. 
\end{proof}
Our next goal is to describe when $B = R(B)$.  The answer to this question depends on the concept of a {\em hereditary} set of vertices.
\begin{definition}
\label{def:hereditary}
    For a directed graph $E$, a subset $H \subseteq E^0$ is {\em hereditary} if $r(s^{-1}(H)) \subseteq H.$    \end{definition}

 Observe that intersections of hereditary subsets are also hereditary. Hence, to each subset \(B\) of \(E^0\) corresponds its \emph{hereditary closure}, which is the smallest hereditary subset of \(E^0\) containing \(B\). 
\begin{proposition}\label{prop.finalsets}
    For a finite graph \(E\) and \(B\subseteq E^0\), the following are equivalent.
    \begin{enumerate}[label=(\arabic*)]
        \item\label{prop.finalsets.image} \(B=R(A)\) for some \(A\subseteq E^0\),
        \item\label{prop.finalsets.arrow} \(B=B^{\rightarrow}\),
        \item\label{prop.finalsets.union} \(B\) is the hereditary closure of a union of sinks and cycles, 
        \item\label{prop.finalsets.fixedpoint} \(B=R(B)\).
    \end{enumerate}
\end{proposition}
\begin{proof}
    Note first that the map \(X\mapsto X^\rightarrow\) preserves set inclusion.

    \begin{enumerate}
        
   \item[\ref{prop.finalsets.image}\(\Rightarrow\)\ref{prop.finalsets.arrow}] Suppose that \(B=R(A)\). If \(x\in B^\rightarrow\), then one of the following hold:
    \begin{itemize}
        \item \(x\) is a sink of \(B\) (so in particular \(x\in B\)).
        \item \(x\) can be flowed into by some point \(b\) of \(B\). 
        But then this \(b\in B=R(A)\) is not a sink, so it can also be flowed into by paths of arbitrarily large size starting at \(A\). Concatenating these paths, we see that \(x\in R(A)=B\).
    \end{itemize}
    
    Conversely, if \(x\in B=R(A)\), then one of the following hold:
    \begin{itemize}
        \item \(x\) is a sink, so \(x\in B^\rightarrow\).
        \item \(x\) is not a sink. Then there are infinitely many paths \(\alpha_1,\alpha_2,\ldots\), all starting in \(A\) and ending at \(x\), with strictly increasing lengths. Let \(x_n\) be the vertex immediately preceding \(x\) in \(\alpha_n\). Since \(E^0\) is finite, some point \(y\) of \(E^0\) will coincide with \(x_n\) for infinitely many \(n\). So, passing to a subsequence if necessary, we can assume that all \(x_n\) are equal to \(y\). Since the paths obtained by dropping the last edge of each \(\alpha_n\) connect some point of \(A\) to \(y\), and these have arbitrarily large lengths, we conclude that \(y\in B = R(A).\) By construction, we have \(x\in r(s^{-1}(y))\subseteq r(s^{-1}(B))\subseteq B^\rightarrow\).
    \end{itemize}
    
   \item[\ref{prop.finalsets.arrow}\(\Rightarrow\)\ref{prop.finalsets.union}] Let \(L\) be the set of all vertices which lie on a cycle contained in \(B\). 
    Then \(L\subseteq B\), so \(L^\rightarrow\subseteq B^\rightarrow=B\). Iterating, we obtain \(L^{m\rightarrow}\subseteq B\) for all \(m\geq 1\). The hereditary closure \(\overline{L}\) of \(L\) is \(\bigcup_{m\geq 1}L^{m\rightarrow}\), so it is contained in \(B\).
    
    It remains to prove that any point \(x\) in \(B\) which is not a sink belongs to \(\overline{L}\). If \(x\in B^\rightarrow\) is not a sink, then there is an edge \(e_1\colon x_1\to x_0=x\) with \(x_1\in B\). But then \(x_1\) is also not a sink, so there is an edge \(e_2\colon x_2\to x_1\) with \(x_2\in B\). Repeating this argument, we find a sequence of edges
    \[\cdots\to x_3\to x_2\to x_1\to x_0=x\]
    with \(x_i\in B\). Since \(E^0\) is finite, we are able to find \(j>i>1\) such that \(x_j=x_i\), and thus we obtain a cycle
    \[x_j\to x_{j-1}\to\cdots x_{i+1}\to x_i=x_j\]
    completely contained in \(L\), so \(x_i\in L\) and \(x\) is, by construction, in the hereditary subset generated by \(x_i\), which is contained in \(\overline{L}\).
    
\item[\ref{prop.finalsets.union}\(\Rightarrow\)\ref{prop.finalsets.fixedpoint}] Under the conditions of \ref{prop.finalsets.union}, it is clear that \(B=B^\rightarrow\), so it follows that \(B=R(B)\).
    
    \item[\ref{prop.finalsets.fixedpoint}\(\Rightarrow\)\ref{prop.finalsets.image}] is trivial.\qedhere
    \end{enumerate}
    
\end{proof}

Next we show that the leaves of a union is the union of the leaves.

\begin{proposition}
    For a finite graph $E$ and $A, B \subseteq E^0$, we have $R(A \cup B) = R(A) \cup R(B).$
\end{proposition}

\begin{proof}
    Clearly, $R(A) \cup R(B) \subseteq R(A \cup B)$; the converse follows from the fact that, by definition,  $(A \cup B)^{ \to} = A^\to \cup B^\to. $ Iterating, it follows that $(A \cup B)^{m \to} = A^{m \to} \cup B^{m \to}$, and hence $R(A \cup B) = R(A) \cup R(B).$ 
\end{proof}

The next theorem states that sets of leaves of a graph \(E\) correspond to Archimedean classes of the respective talented monoid. The following analog of the support of a function will be useful: A representation \(x=\sum_{v\in A,n\in\mathbb{Z}} k_{n,v}({}^nv)\) of an element \(x\in T_E\) is said to be \emph{supported on \(A \subseteq E^0\)} if for each \(v\in A\) there exists \(n\in\mathbb{Z}\) such that \(k_{n,v}\geq 1\). Note that any (nonzero) element of \(T_E\) admits a representation supported on some (nonempty) subset of \(E^0\).

\begin{theorem}\label{leafandarc}
    Let \(E\) be a finite graph and $T_E$ its associated talented monoid. 
    \begin{enumerate}[label=(\arabic*)]
        \item Let $x \in T_E$.  If 
        \[x=\sum_{v\in A, n \in \Z }k_{n,v}({}^{n}v )=\sum_{w\in B, m\in \Z}l_{m,w}{}(^{m}w),\]
        are representations of \(x\) supported on \(A\) and \(B\), respectively, 
        then \(R(A)=R(B)\). We denote this common set of leaves by \(R(x)\).
        \item Any \(x\in T_E\) may be written as a sum of shifts of vertices of \(R(x)\), i.e., as
        \[x=\sum_{v\in R(x), n\in \Z }k_{n,v} {}^{n} v,\]
        with \(k_{n,v}\in \mathbb N\). 
        \item For \(x,y\in T_E\), we have \(x\in\langle y\rangle\) if, and only if, \(R(x)\subseteq R(y)\).
        \item There is a bijection between \(\mathscr{L}(E)\) and the collection of Archimedean classes of \(T_E\): The set of leaves \(A\in\mathscr{L}(E)\) corresponds to the Archimedean class \(
        \left[\sum_{a\in A}a\right]\), and the Archimedean class $[x]$ of \(x\in T_E\) corresponds to the set of leaves \(R(x)\).
    \end{enumerate}
\end{theorem}
\begin{proof}
    {\color{white},}
    \begin{enumerate}[label=(\arabic*)]
        \item By the Confluence Lemma, if we begin with two distinct representations of \(x\), there is a third to which they both flow. Moreover, for any set $A \subseteq E^0$, we have $R(A) = R(A^\to)$. Statement (1) now follows.
        
       
        \item We start with a representation \(x=\sum_{v\in A,n\in\mathbb{Z}}k_{n,v}({}^nv)\) supported on a subset \(A\) of \(E^0\). The definition of \(T_E\) gives us
        \begin{align*}
        x
            &=\sum_{v\in A,n\in\mathbb{Z}} k_{n,v}({}^nv)
                =\sum_{\substack{v\in A,n\in\mathbb{Z}\\v\text{ a sink}}}k_{n,v}({}^nv)
                +\sum_{\substack{v\in A,n\in\mathbb{Z}\\v\text{ not a sink}}}\sum_{e\in s^{-1}(v)}k_{n,v}({}^{n+1}r(e))\\
            &=\sum_{v\in A^{\to},m\in\mathbb{Z}} k'_{m,w}({}^mw),
        \end{align*}
        where \(k'_{m,w}\) is the sum of all terms of the form \(k_{m,w}\) (if \(w\in A\) is a sink) and the form \(k_{m-1,v}\), where \(v\in s(r^{-1}(w))\cap A\). By hypothesis, for all $v\in A$ there exists $n$ so that $k_{n,v} \geq 1$; consequently, for all $w\in A^\to$ there exists $m$ so that $k_{m,w}' \geq 1$.  That is, this is a representation of \(x\) supported on \(A^\to\).
        
        By definition
        \[R(x)=R(A)=\bigcap_{n\geq 1}\left(\bigcup_{m\geq n}A^{m\to}\right).\]
        As \(E^0\) is finite and the intersection above is decreasing, the sequence of sets being intersected stabilizes, i.e., there exists \(M\) such that \(\bigcup_{m\geq n}A^{m\to}=R(x)\) for all \(n\geq M\). Iterating the procedure in the previous paragraph, we can find a representation of \(x\) supported on \(A^{M\to}\), which satisfies the required properties.

        \item If \(x\in\langle y\rangle\), then in \(M_E\) we have \(x\leq ny\) for some \(n\in\mathbb{N}\). By the Confluence Lemma on \(E\), we can write \(x\) and \(y\) in \(M_E\) in the form
        \[x=\sum_{v\in E^0}k(v)v\qquad\text{and}\qquad y=\sum_{v\in E^0}k'(v)v\]
        where \(0 \leq k(v)\leq k'(v)\). By definition,
        \[R(x)=R(\left\{v\mid k(v)\geq 1\right\})\]
        and similarly for \(R(y)\), so \(R(x)\subseteq R(y)\).
        
        Conversely, suppose that \(R(x)\subseteq R(y)\). Then any \(v\in R(x)\subseteq R(y)\) is flowed into by some vertex which appears in a representation of \(y\) in \(T_E\).  That is, \(v\leq {}^jy\) for some  \(j\in \Z\). 
        
        Therefore, all \(v\in R(x)\), seen as elements of \(T_E\), belong to \(\langle y\rangle\). Taking shifts and sums, the previous item allows us to conclude that \(x\in\langle y\rangle\).
        \item By the previous item, the set of leaves of \(x\) depends only on the Archimedean class of \(x\), so the map \([x]\mapsto R(x)\) is well-defined.
        For \(A\in\mathscr{L}(E)\), it is clear that \(R(\sum_{a\in A}a)=R(A)=A\).

        Conversely, let \(x\in T_E\) and \(\overline{x}=\sum_{v\in R(x)}v\). We need to prove that \(x\) and \(\overline{x}\) belong to the same Archimedean class. From (b), we know that \(x\) may be written as a sum of shifts of vertices \({}^iv\) with \(v\in R(x)\). Since \({}^iv\leq {}^i\overline{x}\),  each \({}^iv\) belongs to \(\langle \overline{x}\rangle\), and so does their sum, \(x\).
        
        On the other hand, the definition of \(R(x)\) gives that any \(v\in R(x)\) is flowed into by some vertex in some representation of \(x\). This implies that, for \(v\in R(x)\), we have \(v\leq{}^ix\) for some \(i\), so \(v\in\langle x\rangle\). Summing over \(v\in R(x)\) yields \(\overline{x}\in\langle x\rangle\).\qedhere
    \end{enumerate}
\end{proof}

Now, let \(M\) be a monoid which, in principle, does not carry a \(\mathbb{Z}\)-action. Then, Archimedean classes of \(M\) are defined as if \(M\) was given the trivial \(\mathbb{Z}\)-action: \({}^n x=x\) for all \(n\in\mathbb{Z}\) and \(x\in M\). This, in particular, applies to the graph monoid \(M_E\) of a row-finite graph \(E\).  This leads to the following Corollary.

\begin{corollary}
    For a finite graph \(E\), the map
    \[A\mapsto \sum_{a\in A}a\]
    establishes a bijection between the collection of sets of leaves of \(E\) and the collection of Archimedean classes of \(M_E\).
\end{corollary}
\begin{proof}
   We use the bijection of Theorem \ref{leafandarc} between $\mathscr L(E)$ and the collection of Archimedean classes of $T_E$.  Observe that the map $T_E \to M_E$ given by 
\[{}^nv\mapsto v\quad (v\in E^0,\quad n\in\mathbb{Z})\]
induces a bijection between the collections of $\Z$-Archimedean classes of \(T_E\) (with its usual \(\mathbb{Z}\)-monoid structure) and of \(M_E\) (as a monoid). So, the geometric characterization of Archimedean classes from Theorem \ref{leafandarc} also applies to \(M_E\).
\end{proof}

\subsection{Meteor graphs and their Archimedean classes}
\label{sec:arch-for-meteor}

In this section, we specialize our study of Archimedean classes to the case of meteor graphs, which are the objects of interest in this paper.
Before we define meteor graphs, recall that a directed graph such that every vertex lies on at most one cycle is called a \emph{graph with disjoint cycles}. Clearly, in such a graph any two distinct cycles do not have a common vertex.    For cycles $C_1$, $C_2$ in a graph $E$, we write $C_1 \Rightarrow C_2$ if there exists a path that starts in $C_1$ and ends in $C_2$. 
For the class of graphs with disjoint cycles, $\Rightarrow$ is a reflexive, transitive, antisymmetric relation. 

\begin{definition}
    \label{def:meteor}
A \emph{meteor graph} is a connected essential graph $E$ 
consisting of two disjoint cycles $C_0^E, C_1^E$ {and the paths connecting these cycles}. 

In a meteor graph $E$, the two cycles must necessarily be related by  $C_0^E \Rightarrow C_1^E$.  We call $C_0^E$ the {\em source cycle} of $E$, and $C_1^E$ the {\em sink cycle} of $E$.  

 A path $\alpha=p_0 p_1 \cdots p_n$ in a meteor graph $E$ is called a \emph{trail} if $s(p_0) \in (C^E_0)^0$, $r(p_n) \in (C^E_1)^0$, and $s(p_i) \not \in (C^E_0)^0 \cup (C^E_1)^0$ for $1\leq i \leq n$.
 For a trail $\alpha$, its {\em interior} is $p^0 \backslash ((C^E_0)^0 \cup (C^E_1)^0)$. Note that the interior could be an empty set. 
\end{definition}

The following result can be extended to graphs with disjoint cycles, but we will write it in the special case of meteor graphs, which gives a hint of how the talented monoid will be used in Theorem \ref{thm:main} to establish our main results on symbolic dynamics. 

\begin{proposition}\label{unipresen} 
Let $E$ be a meteor graph, containing cycles $C_0^E$ and $C_1^E$ with $C_0^E \Rightarrow C_1^E$. Let $v$ and $w$ be vertices on $C_0^E$ and $C_1^E$, respectively. Then for any element $x\in T_E$, there exists $j\in \mathbb Z$ such that 
$x$ can be uniquely written as 
\[\sum_{0\leq i< p} a_iv(j-i) +  \sum_{0\leq i< q} b_iw(i),\]
where $|C_0^E|=p$, $|C_1^E|=q$ and $a_i,b_i \in \mathbb N$.  
    
\end{proposition}
\begin{proof}
It is easy to observe that we can write 
\begin{equation}\label{genxy}
x = \sum_{i=0}^n c_i v(r_i) + \sum_{i=0}^m d_i w(s_i),
\end{equation} for some $c_i, d_i \in \mathbb N$ and  $r_i, s_i \in \mathbb Z$. 
Moreover, for any vertex $z$ in the interior of a trail of $E$, there exists $j \in \N$ such that $z(j) = w$.  Consequently, $z = z(0) = w(-j)$.  Using this observation, and the facts that $C_0^E \Rightarrow C_1^E$ and $ |C_0^E|=p,$ 
we conclude that 
\begin{equation}\label{enjoymil}
v = v(p)+\sum_l e_l w(t_l),  
\end{equation} where $e_l \in \mathbb N$ and $t_l \in \mathbb Z$. Without loss of generality, we can assume that $r_n$ is the largest integer among $\{r_1,\dots,r_n\}$. For any $i$, write $r_n - r_i = pl_i +r'_i$, where $0 \leq r'_i < p$. Thus $r_n - r'_i= r_i+pl_i$. By repeated use of~(\ref{enjoymil}), we can write 
\[v(r_i) = v(r_i + pl_i ) + \sum_l e'_l w(t'_l) =  v(r_n- r'_i ) + \sum_l e'_l w(t'_l). \]
Substituting now into the identity~(\ref{genxy}) we obtain 
\begin{equation*}
x = \sum_{i=0}^n c_i v(r_n-r'_i) + \sum_i d'_i w(s'_i).
\end{equation*}

Setting $j:= r_n$, since $0\leq r_i' < p $, we can write 
\[x = \sum_{0\leq i< p} a_iv(j-i) + \sum_i d'_i w(s'_i).\]

On the other hand, since $C_1^E$ is a cycle of length $q$ with no exit,  
$w=w(q)$ in $T_E$. Thus $w(s)=w(s')$ whenever $s \equiv s' \pmod q$. This allows us to re-write $x$ as  
\[x = \sum_{0\leq i< p} a_iv(j-i) +  \sum_{0\leq i< q} b_iw(i).\]

Now we show that this presentation is unique. Suppose that we can also write 
\[x = \sum_{0\leq i< p} a'_iv(j-i) +  \sum_{0\leq i< q} b'_iw(i).\]
We will show that $a_i=a'_i$ for all $0\leq i <p$ and $b_i=b'_i$ for all $0\leq i <q$ . 

Consider the $\mathbb Z$-order ideal $\langle w \rangle$ generated by $w$ (cf.~Equation \eqref{eq:order-ideal}). The fact that $v(j) \leq {}^p v(j)  $, for any $j\in \Z$, implies that $v(j) \not\in \langle w \rangle$.  Indeed, the natural inclusion $T_{C_0^E} \hookrightarrow T_E$ descends to an isomorphism $T_E / \langle w \rangle \cong T_{C_0^E}$. The images of the presentations of $x$ in this quotient monoid are 
\[ \sum_{0\leq i< p} a'_iv(j-i) = \sum_{0\leq i< p} a_iv(j-i).\]
Since $T_{C_0^E}$ is a cyclic monoid (see Definition~\ref{cycmond}), it follows that $a_i =a'_i$, for all $i$'s. 
Since $T_E$ is cancellative, we can remove these portions from the representation of $x$ and obtain 
\begin{equation}\label{hghghg1}
  \sum_{0\leq i< q} b_iw(i)= \sum_{0\leq i< q} b'_iw(i) \in T_E.  
\end{equation} 
Since the cycle $C_1^E$ containing $w$ has no exit,   the vertices $(C_1^E)^0$ form a hereditary subset of $E$. Consider the monoid $T_{C_1^E}$. Using the Confluence Lemma~\ref{confuu}, and the fact that $(C^E_1)^0$ is hereditary, 
it is easy to check that the canonical map 
$T_{C_1^E}\rightarrow T_E$ is injective. Therefore, the equality (\ref{hghghg1}) also holds in 
$T_{C_1^E}$. However the monoid $T_{C_1^E}$ is cyclic, and thus it follows that $b_i =b'_i$, for all $i$'s.
\end{proof}

We now come to the main Theorem of this section, in which we describe the Archimedean components of a meteor graph.

\begin{theorem}
    \label{ghostagain}

Let $E$ be a meteor graph with $C_0^E \Rightarrow C_1^E$, $v\in (C_0^E)^0$, $w\in (C_1^E)^0$ and $p=|C_0^E|$, $q=|C_1^E|$.  Let $T_E$ be the associated talented monoid. Then, $T_E$ can be decomposed as the disjoint union of its Archimedean $\Z$-semigroup components:
\[T_E = \{0\} \, \dot \cup \, [v] \, \dot \cup\,  [w],\] where 
\begin{align*}
[w] & = \Big \{ \sum_{0\leq i \leq q} k_i w(i)\mid k_i \in \mathbb N,  \exists i \text{ s.t. } k_i \not = 0  \Big \}, \\
[v] & = \Big \{ \sum_i l_i v(i) + \sum_{0\leq j \leq q} k_j w(j)  \mid  l_i, k_j \in \mathbb N, \exists i \text{ s.t. } l_i \not = 0 \Big \}. \\
\end{align*}
\end{theorem}
\begin{proof}

By Theorem~\ref{leafandarc}(4), there is a bijection between $\mathscr{L}(E)= \{R(A)\mid A\subseteq E^0\} $ and the Archimedean classes of $T_E$. We first observe that, for a nonempty set $A\subseteq E^0$, if $A \cap (C_0^E)^0 \not = \emptyset$ then $R(A)=R(v)=E^0$, and $R(A)=R(w)=(C^E_1)^0$ otherwise. 
To see this, suppose first that $z \in A \cap (C_0^E)^0 \not = \emptyset$. By Proposition~\ref{helpprop}, $(C_0^E)^0 \cup (C_1^E)^0\subseteq R(z)$. If $w$ is a vertex on a trail between $C_0^E$ and $C_1^E$, then there is a path  $\gamma$  from $z$ to $w$. 
Notice that for any $m \geq 0$, $w \in \{z\}^{(m\cdot p+ |\gamma| )\to} $. Thus, $w\in R(z)$ as well, and consequently   $R(z)=E^0$. Since $R(z)\subseteq R(A)$, it follows that $R(A)=E^0$. 

On the other hand, if $A \cap (C_0^E)^0  = \emptyset$, then all elements of $A$ have to be either on the trails between $C_0^E$ and $C_1^E$, or on $(C_1^E)^0$. Again, Proposition~\ref{helpprop} implies that $(C_1^E)^0\subseteq R(A)$. Let $l$ be the length of the longest trail between $C_0^E$ and $C_1^E$. Since 
$C_0\Rightarrow C_1$, it follows that  $A^{k\rightarrow}\subseteq (C_1^E)^0$, for any $k\geq l$. Thus $R(A)\subseteq (C_1^E)^0$ as well, and consequently $R(A)=(C_1^E)^0$ if $A \cap (C_0^E)^0 = \emptyset$. 

Now, from Theorem~\ref{leafandarc}(4) we obtain that the Archimedean classes of $T_E$ are given by $[v]$ and $[w]$ (and also $\{0\}$, corresponding to the empty set). This gives the first part of the Theorem. The description of $[w]$ and $[v]$ follows from Proposition~\ref{unipresen}.
\end{proof}

\section{The dynamics of meteor graphs}\label{meteorchap}

In this section, we study the behavior of meteor graphs under the graph moves of in- and out-splitting and prove our main result, Theorem \ref{thm:main}. We start with the following definition which allows us to describe meteor graphs differently.   

\begin{definition}
A cycle $\alpha$ in a directed graph $E$  is a \emph{source cycle} if $|r^{-1}(v)| = 1$ for all $v\in \alpha^0$. \emph{Sink cycles} are defined analogously, as the cycles $\alpha$ where $|s^{-1}(v) |=1$ for all $v\in \alpha^0$. 
\end{definition}

\begin{example}

The graph below is a graph with disjoint cycles, consisting of three cycles: \(e_1e_2e_3\) is a source cycle; \(g_1g_2g_3g_4\) is a sink cycle; \(f_1f_2\) is neither.
 
\end{example}

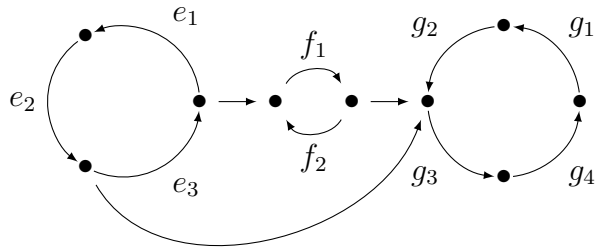
\begin{figure}[!ht]
    \centering
    \begin{tikzpicture}
        \node (c1) at (0,0) {};
        \node (p11) at ([shift=(0:1)]c1) {\(\bullet\)};
        \node (p12) at ([shift=(120:1)]c1) {\(\bullet\)};
        \node (p13) at ([shift=(240:1)]c1) {\(\bullet\)};
        \draw[-latex] (0,0) ++(7:1) arc (7:113:1) node[midway,above right] {\(e_1\)};
        \draw[-latex] (0,0) ++(127:1) arc (127:233:1) node[midway,left] {\(e_2\)};
        \draw[-latex] (0,0) ++(247:1) arc (247:353:1) node[midway,below right] {\(e_3\)};

        \node (u1) at ([shift=(0:1)]p11) {\(\bullet\)};
        
        \node (u2) at ([shift=(0:1)]u1) {\(\bullet\)};
        
        \node (u3) at ([shift=(0:1)]u2) {\(\bullet\)};
        
        \node (c2) at ([shift=(0:1)]u3) {};
        \node (p21) at ([shift=(90:1)]c2) {\(\bullet\)};
        \node (p22) at ([shift=(0:1)]c2) {\(\bullet\)};
        \node (p23) at ([shift=(-90:1)]c2) {\(\bullet\)};
        \draw[-latex] (c2) ++(7:1) arc  (7:83:1) node[midway,above right] {\(g_1\)};
        \draw[-latex] (c2) ++(97:1) arc (97:173:1) node[midway,above left] {\(g_2\)};
        \draw[-latex] (c2) ++(187:1) arc (187:263:1) node[midway,below left] {\(g_3\)};
        \draw[-latex] (c2) ++(277:1) arc (277:353:1) node[midway,below right] {\(g_4\)};

        \draw[-latex] (u1) to[out=60,in=120] node[midway,above] {\(f_1\)} (u2) ;
        \draw[-latex] (u2) to[out=-120,in=-60] node[midway,below] {\(f_2\)} (u1);
        \draw[-latex] (u2) to (u3);
        
        \draw[-latex] (p11)--(u1);
        \draw[-latex] (p13) to[out=-60,in=250] (u3);
    \end{tikzpicture}
 \label{fig:cycles}
\end{figure}
\begin{center}
    
\end{center}


Next, we observe that the class of meteor
graphs is closed under the notion of shift equivalence (see Section \ref{dynref}), using Corollary \ref{h99}. 
For the proof of this proposition, we need a few definitions.

\begin{definition}
    \label{def:saturated}
    Let $E$ be a directed graph.
    A subset $H \leq E^0$ is {\em saturated} if $r(s^{-1}(v)) \subseteq H$ implies $v \in H$.
    Given a set $X \subseteq E^0$, its {\em saturation} $\Sigma X$ is the smallest hereditary saturated subset containing $X$.
\end{definition}
It is easy to prove \cite[Lemma~2.2]{hazli} that if $H \leq E^0$ is hereditary and saturated, and we write $E \backslash H$ for the graph $(E^0 \backslash H, \{ e \in E^1: r(e) \not\in H\}, r_E, s_E)$, then $T_E /\langle H \rangle \cong T_{E \backslash H}$.  Moreover, for any $X \subseteq E^0$, we have $\langle X\rangle = \langle \Sigma X \rangle$ as ideals of $T_E$.

For a $\Gamma$-monoid $M$, recall from \cite{hazli} that an element $0 \not = a\in M$ is \emph{periodic} if there is an $0 \not = \alpha \in \Gamma$ such that ${}^\alpha a =a$. If $a\in M$ is not periodic, we call it \emph{aperiodic}. We denote the orbit of the action of $\Gamma$ on an element $a$ by $O(a)$, so $O(a)=\{{}^\alpha a \mid \alpha \in \Gamma \}$.  Furthermore, in a  conical and cancellative $\Gamma$-monoid $M$, we say an element $a$ is \emph{minimal} if $0\not = b\leq a$ then $a=b$.

In the following Proposition, we will use that the set of cycles without exits in a graph $E$ is in one-to-one correspondence with the orbits of minimal periodic elements of the talented monoid $T_E$, \cite[Lemma~5.6]{hazli}.

\begin{proposition}\label{goldenprop}
Let $E$ and $F$ be essential graphs that are shift equivalent. Then $E$ is a meteor graph if and only if $F$ is a meteor graph. 
\end{proposition}

\begin{proof}
    Since $E$ and $F$ are shift equivalent, by Corollary ~\ref{h99},  there is an order-preserving $\mathbb Z[x,x^{-1}]$-module  isomorphism $K_0^{\gr}(L(E)) \cong K_0^{\gr}(L(F))$. Since the positive cone of the graded Grothendieck group $K_0^{\gr}(L(E))$ is $\mathbb Z$-isomorphic to the talented monoid $T_E$~(see \cite{hazli}), it follows that there is a $\mathbb Z$-monoid isomorphism $\phi: T_E\rightarrow T_F$. Let $E$ be a meteor graph with source cycle $C^E_0$ and sink cycle $C^E_1$. Let $w_E\in (C_1^E)^0$. Then $w_E$ is a minimal periodic element in $T_E$, i.e.,  ${}^q w_E =w_E$, where $|C^E_1|=q.$ It follows that  $\phi(w_E)$ is a {minimal} periodic element in $T_F$. Thus $\phi(w_E)={}^i w_F$ for some vertex $w_F$  on a cycle in $F$ without exit and some $i\in \mathbb N$. Call this cycle without exit $C_1^F$. Let $\langle w_E \rangle$ be the $\mathbb Z$-order ideal generated by $w_E$. We have $\phi(\langle w_E \rangle)=\langle w_F \rangle$. Thus, passing the isomorphism $\phi$ to the quotients, we have the induced $\mathbb Z$-monoid isomorphism $\overline \phi: T_E/\langle w_E \rangle \rightarrow T_F/\langle w_F \rangle$. But $T_E/\langle w_E \rangle \cong T_{E}\backslash \Sigma \{ w_E\}$.
Since $E$ is a meteor graph, $E \backslash \Sigma \{ w_E\} 
\cong C_0^E$ is a cycle without exit. 
Thus $T_E/\langle w_E \rangle$ is a $\mathbb Z$-cyclic 
monoid and hence so is  $T_F/\langle w_F \rangle$. Thus, $F\backslash \Sigma \{ w_F \}$
consists of one cycle without exit, call it $C_0^F$. This implies that $F$ contains exactly two disjoint cycles $C_0^F$ and $C_1^F$.  As we assumed $F$ was essential, and hence connected, we conclude that $F$ is a meteor graph.
\end{proof}

   



In what follows, we show that any meteor graph can be transformed, using only in- and out-splits and -amalgamations, into a meteor graph in which all trails have the same source and consist of exactly one edge. We formalize the concepts involved below.

\begin{definition}
    \label{def:normal-form}
A meteor graph $E$ is in {\em quasi-normal form} if all trails of $E$ consist of exactly one edge.
    A meteor graph $E$ is in {\em normal form} if it is in quasi-normal form and all of the trails of $E$ have the same source.
\end{definition}

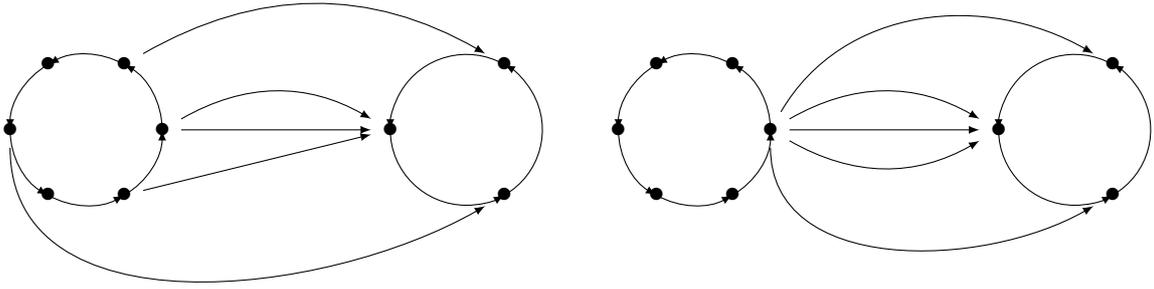
\begin{figure}[!ht]
    \centering
    \begin{tikzpicture}
        \node (center1) at (0,0) {};
        \node (center2) at ([shift={(0:5)}]center1) {};
        \node (a1) at ([shift=(0:1)]center1) {\(\bullet\)};
        \node (a2) at ([shift=(60:1)]center1) {\(\bullet\)};
        \node (a3) at ([shift=(120:1)]center1) {\(\bullet\)};
        \node (a4) at ([shift=(180:1)]center1) {\(\bullet\)};
        \node (a5) at ([shift=(240:1)]center1) {\(\bullet\)};
        \node (a6) at ([shift=(300:1)]center1) {\(\bullet\)};
        \node (b1) at ([shift={(180:1)}]center2) {$\bullet$};
        \node (b2) at ([shift={(60:1)}]center2) {$\bullet$};
        \node (b3) at ([shift={(300:1)}]center2) {$\bullet$};

        \draw[-latex] (a1) arc (0:60:1);
        \draw[-latex] (a2) arc (60:120:1);
        \draw[-latex] (a3) arc (120:180:1);
        \draw[-latex] (a4) arc (180:240:1);
        \draw[-latex] (a5) arc (240:300:1);
        \draw[-latex] (a6) arc (300:360:1);
        \draw[-latex] (b1) arc (-180:-60:1);
        \draw[-latex] (b2) arc (60:180:1);
        \draw[-latex] (b3) arc (-60:60:1);
        
        \draw[-latex] (a1) to (b1);
        \draw[-latex] (a1) to [out=30,in=150] (b1);
        \draw[-latex] (a2) to [out=30,in=150] (b2);
        \draw[-latex] (a6) to (b1);
        \draw[-latex] (a4) to [out=270,in=210] (b3);
        
        \node (center1) at (8,0) {};
        \node (center2) at ([shift={(0:5)}]center1) {};
        \node (a1) at ([shift=(0:1)]center1) {\(\bullet\)};
        \node (a2) at ([shift=(60:1)]center1) {\(\bullet\)};
        \node (a3) at ([shift=(120:1)]center1) {\(\bullet\)};
        \node (a4) at ([shift=(180:1)]center1) {\(\bullet\)};
        \node (a5) at ([shift=(240:1)]center1) {\(\bullet\)};
        \node (a6) at ([shift=(300:1)]center1) {\(\bullet\)};
        \node (b1) at ([shift={(180:1)}]center2) {$\bullet$};
        \node (b2) at ([shift={(60:1)}]center2) {$\bullet$};
        \node (b3) at ([shift={(300:1)}]center2) {$\bullet$};

        \draw[-latex] (a1) arc (0:60:1);
        \draw[-latex] (a2) arc (60:120:1);
        \draw[-latex] (a3) arc (120:180:1);
        \draw[-latex] (a4) arc (180:240:1);
        \draw[-latex] (a5) arc (240:300:1);
        \draw[-latex] (a6) arc (300:360:1);
        \draw[-latex] (b1) arc (-180:-60:1);
        \draw[-latex] (b2) arc (60:180:1);
        \draw[-latex] (b3) arc (-60:60:1);
        
        \draw[-latex] (a1) to (b1);
        \draw[-latex] (a1) to [out=30,in=150] (b1);
        \draw[-latex] (a1) to [out=60,in=150] (b2);
        \draw[-latex] (a1) to [out=-30,in=210] (b1);
        \draw[-latex] (a1) to [out=270,in=210] (b3);
        
        
    \end{tikzpicture}\caption{A graph in quasi-normal (not normal) form and a graph in normal form, res\-pec\-ti\-ve\-ly.}
    \label{fig:normal_quasi}
\end{figure}

\begin{proposition}
\label{prop:convert-to-normal}
    Any meteor graph  can be transformed, using only in- and out-splits and 
    -amalgamations, into a meteor graph in normal form.

\end{proposition}
\begin{proof}
   
    Consider an arbitrary meteor graph \(E\). As a first step, let us transform \(E\) into another meteor graph -- using the previously mentioned graph moves  -- into another meteor graph where the interiors of distinct trails do not intersect.
    

    Assume that there are distinct trails of \(E\) whose interiors intersect. That is, there is some vertex \(u\) in the interior of a trail which is the source or range of more than one edge. Assuming the prior case, let us temporarily call such a vertex a ``multisource vertex''. Performing an out-split at \(u\) and separating the edges that come out of it, we obtain new vertices \(u_1,u_2,\ldots\) which are sources of only one edge each. See Figure \ref{fig.multisource.vertex}.
    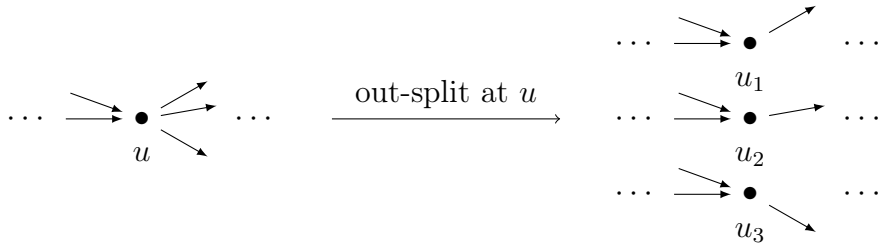
\begin{figure}
        \centering
        \begin{tikzpicture}
     \node [label={[label distance=0cm]270:$u$}] (u) at (0,0) {$\bullet$};
     \draw[-latex] (u) to ([shift=({(30:1)})]u);
     \draw[-latex] (u) to ([shift=({(10:1)})]u);
     \draw[-latex] (u) to ([shift=({(-30:1)})]u);
     \draw[-latex] ([shift=({160:1})]u.center) to (u);
     \draw[-latex] ([shift=({180:1})]u.center) to (u);
     \node at ([shift=(0:1.5)]u.center) {$\cdots$};
     \node at ([shift=(180:1.5)]u.center) {$\cdots$};

     \node (center2) at (8,0) {};
     \draw[->] (2.5,0) to node[midway,above] {out-split at $u$} ([shift={(-2.5,0)}]center2) ;
     \node [label={[label distance=0cm]270:$u_1$}] (u1) at ([shift={(0,1)}]center2) {$\bullet$};
     \node [label={[label distance=0cm]270:$u_2$}] (u2) at (center2) {$\bullet$};
     \node [label={[label distance=0cm]270:$u_3$}] (u3) at ([shift={(0,-1)}]center2) {$\bullet$};
     \draw[-latex] (u1) to ([shift=({(30:1)})]u1);
     \draw[-latex] (u2) to ([shift=({(10:1)})]u2);
     \draw[-latex] (u3) to ([shift=({(-30:1)})]u3);
     \foreach \i in {1,2,3} {
         \draw[-latex] ([shift=({160:1})]u\i.center) to (u\i);
         \draw[-latex] ([shift=({180:1})]u\i.center) to (u\i);
         \node at ([shift=(0:1.5)]u\i.center) {$\cdots$};
         \node at ([shift=(180:1.5)]u\i.center) {$\cdots$};
     }
     \end{tikzpicture}
        \caption{Out-split at a multisource}
        \label{fig.multisource.vertex}
    \end{figure}
    
    The trails in this new graph have the same length as the trails in the original graph. The multisource vertices in this new graph are the same ones as in the original graph, except for \(u\), plus all vertices in \(s(r^{-1}(u))\) are now multisource vertices. 
    Thus, performing out-splits at all multi-sources, starting at the ones closest to the sink cycle and moving to the ones closer to the source cycle, gets rid of all of them. In this manner, we may transform our original graph into one in which all vertices in the interior of the trails are only sources of one edge each.
    
    We then perform in-splits in the interior of a trail (using a partition with exactly one edge in each set), and further transform the graph into one such that each vertex in the interior of the trails is the range of exactly one edge. In this manner, the resulting graph's trails have non-intersecting interiors. See Figure \ref{fig.removing.intersections} for an example.
    
    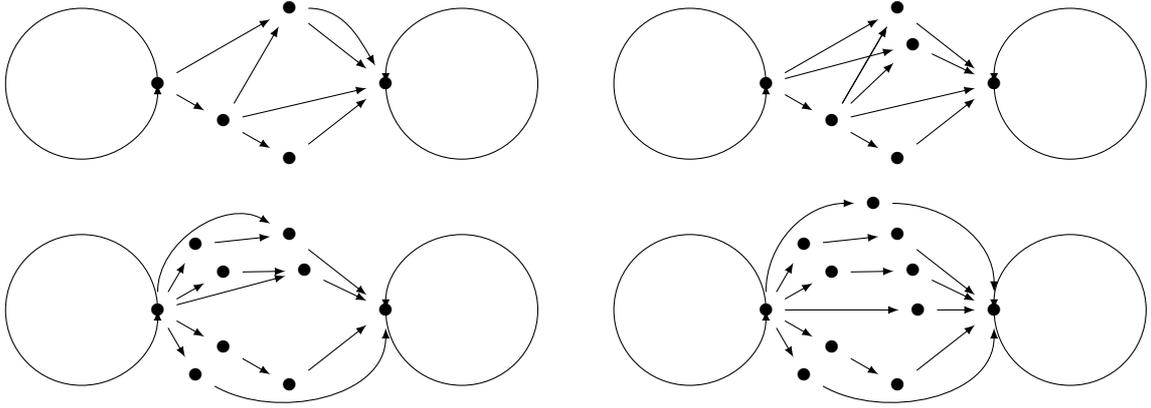
\begin{figure}[!ht]\centering
    \begin{tikzpicture}
        \node (center) at (0,0) {};
        \node (a) at ([shift=(0:1)]center) {\(\bullet\)};
        \node (u) at ([shift={(-30:1)}]a) {$\bullet$};
        \node (v1) at ([shift={(30:2)}]a) {$\bullet$};
        \node (v2) at ([shift={(-30:2)}]a) {$\bullet$};
        \node (b) at ([shift={(0:3)}]a) {$\bullet$};

        \draw[-latex] (a) arc (0:360:1);
        \draw[-latex] (b) arc (-180:180:1);
        
        \draw[-latex] (a) to (u);
        \draw[-latex] (a) to (v1);
        \draw[-latex] (u) to (v1);
        \draw[-latex] (u) to (v2);
        \draw[-latex] (u) to (b);
        \draw[-latex] (v1) to (b);
        \draw[-latex] (v1) to [out=0,in=120] (b);
        \draw[-latex] (v2) to (b);
        
        \node (center) at (8,0) {};
        
        \node (a) at ([shift=(0:1)]center) {\(\bullet\)};
        \node (u) at ([shift={(-30:1)}]a) {$\bullet$};
        \node (v11) at ([shift={(30:2)}]a) {$\bullet$};
        \node (v12) at ([shift={(15:2)}]a) {$\bullet$};
        \node (v2) at ([shift={(-30:2)}]a) {$\bullet$};
        \node (b) at ([shift={(0:3)}]a) {$\bullet$};

        \draw[-latex] (a) arc (0:360:1);
        \draw[-latex] (b) arc (-180:180:1);
        
        \draw[-latex] (a) to (u);
        \draw[-latex] (a) to (v11);
        \draw[-latex] (a) to (v12);
        \draw[-latex] (u) to (v11);
        \draw[-latex] (u) to (v11);
        \draw[-latex] (u) to (v12);
        \draw[-latex] (u) to (v2);
        \draw[-latex] (u) to (b);
        \draw[-latex] (v11) to (b);
        \draw[-latex] (v12) to (b);
        \draw[-latex] (v2) to (b);
        
        \node (center) at (0,-3) {};
        
        \node (a) at ([shift=(0:1)]center) {\(\bullet\)};
        \node (u1) at ([shift={(60:1)}]a) {$\bullet$};
        \node (u2) at ([shift={(30:1)}]a) {$\bullet$};
        \node (u3) at ([shift={(-30:1)}]a) {$\bullet$};
        \node (u4) at ([shift={(-60:1)}]a) {$\bullet$};
        \node (v11) at ([shift={(30:2)}]a) {$\bullet$};
        \node (v12) at ([shift={(15:2)}]a) {$\bullet$};
        \node (v2) at ([shift={(-30:2)}]a) {$\bullet$};
        \node (b) at ([shift={(0:3)}]a) {$\bullet$};

        \draw[-latex] (a) arc (0:360:1);
        \draw[-latex] (b) arc (-180:180:1);
        
        \draw[-latex] (a) to (u1);
        \draw[-latex] (a) to (u2);
        \draw[-latex] (a) to (u3);
        \draw[-latex] (a) to (u4);
        \draw[-latex] (a) to [out=90,in=150] (v11);
        \draw[-latex] (a) to (v12);
        \draw[-latex] (u1) to (v11);
        \draw[-latex] (u2) to (v12);
        \draw[-latex] (u3) to (v2);
        \draw[-latex] (u4) to [out=-30,in=-90] (b);
        \draw[-latex] (v11) to (b);
        \draw[-latex] (v12) to (b);
        \draw[-latex] (v2) to (b);
        \node (center) at (8,-3) {};
        
        \node (a) at ([shift=(0:1)]center) {\(\bullet\)};
        \node (u1) at ([shift={(60:1)}]a) {$\bullet$};
        \node (u2) at ([shift={(30:1)}]a) {$\bullet$};
        \node (u3) at ([shift={(-30:1)}]a) {$\bullet$};
        \node (u4) at ([shift={(-60:1)}]a) {$\bullet$};
        \node (v111) at ([shift={(45:2)}]a) {$\bullet$};
        \node (v112) at ([shift={(30:2)}]a) {$\bullet$};
        \node (v121) at ([shift={(15:2)}]a) {$\bullet$};
        \node (v122) at ([shift={(0:2)}]a) {$\bullet$};
        \node (v2) at ([shift={(-30:2)}]a) {$\bullet$};
        \node (b) at ([shift={(0:3)}]a) {$\bullet$};

        \draw[-latex] (a) arc (0:360:1);
        \draw[-latex] (b) arc (-180:180:1);
        
        \draw[-latex] (a) to (u1);
        \draw[-latex] (a) to (u2);
        \draw[-latex] (a) to (u3);
        \draw[-latex] (a) to (u4);
        \draw[-latex] (a) to [out=90,in=180] (v111);
        \draw[-latex] (a) to (v122);
        \draw[-latex] (u1) to (v112);
        \draw[-latex] (u2) to (v121);
        \draw[-latex] (u3) to (v2);
        \draw[-latex] (u4) to [out=-30,in=-90] (b);
        \draw[-latex] (v111) to [out=0,in=90](b);
        \draw[-latex] (v112) to (b);
        \draw[-latex] (v121) to (b);
        \draw[-latex] (v122) to (b);
        \draw[-latex] (v2) to (b);

        
    \end{tikzpicture}\caption{Out-splits in the interior vertices, starting from those closest to the sink cycle and moving left, followed by in-splits in the opposite direction. In simple terms, this procedure turns each original trail into a new trail with the same source, range, and length, but the new trails have pairwise disjoint interiors.}\label{fig.removing.intersections}\end{figure}
    
    Now, let $v$ be a  vertex in the source cycle which is the source of a trail.
    Performing the out-split which separates precisely the source-most edge in the trail from \(s^{-1}(v)\) is the same as 
    moving the trail's source back one step along the source cycle, and lengthening the trail by one edge.  See Figure \ref{fig.insplit.sink.loop} below.
    
    \begin{figure}[!ht]\centering
    \begin{tikzpicture}
        \node (center) at (0,0) {};
        \node (u) at ([shift=(-45:1)]center) {\(\bullet\)};
        \node (v) at ([shift=(0:1)]center) {\(\bullet\)};
        \node (w) at ([shift=(45:1)]center) {\(\bullet\)};

        \draw[-latex] (0,0) ++(-75:1) arc (-75:-45:1);
        \draw[-latex] (u) arc (-45:0:1) node[left] {$v$};
        \draw[-latex] (v) arc (0:45:1);
        \draw[-latex] (w) arc (45:75:1);
        \draw[-latex,dashed] (u) to ([shift=(-45:1)]u.center) node[right] {$\cdots$};
        \draw[-latex] (v) to ([shift=(-30:1)]v.center) node[right] {$\cdots$};
        \draw[-latex,dashed] (v) to ([shift=(30:1)]v.center) node[right] {$\cdots$};
        \draw[-latex,dashed] (w) to ([shift=(45:1)]w.center) node[right] {$\cdots$};
        
        \draw[dotted] (center) ++ (80:1) arc (80:280:1);
        
        \node (center2) at (7,0) {};
        \draw [->] ([shift={+(3,0)}]center.center)--([shift={(-1.5,0)}]center2.center) node[midway,above] {out-split at $v$};
        
        \node (uu) at ([shift=(-45:1)]center2) {\(\bullet\)};
        \node (v1) at ([shift=(0:1)]center2) {\(\bullet\)};
        \node (ww) at ([shift=(45:1)]center2) {\(\bullet\)};
        \node (v2) at ([shift=(0:2)]center2) {\(\bullet\)};
        
        \draw[-latex] (center2) ++(-75:1) arc (-75:-45:1);
        \draw[-latex] (uu) arc (-45:0:1) node[left] {$v_1$};
        \draw[-latex] (uu) to (v2) node[above right] {$v_2$};
        \draw[-latex] (v1) arc (0:45:1);
        \draw[-latex] (ww) arc (45:75:1);
        
        \draw[-latex,dashed] (uu) to ([shift=(-45:1)]uu.center) node[right] {$\cdots$};
        \draw[-latex] (v2) to ([shift=(-30:1)]v2.center) node[right] {$\cdots$};
        \draw[-latex,dashed] (v1) to ([shift=(30:1)]v1.center) node[right] {$\cdots$};
        \draw[-latex,dashed] (ww) to ([shift=(45:1)]ww.center) node[right] {$\cdots$};
        
        \draw[dotted] (center2) ++ (80:1) arc (80:280:1);
    \end{tikzpicture}\caption{Out-split at a vertex in the source cycle. The dashed arrows represent trails that may or may not exist, but nevertheless are not modified by the graph move.}\label{fig.insplit.sink.loop}\end{figure}
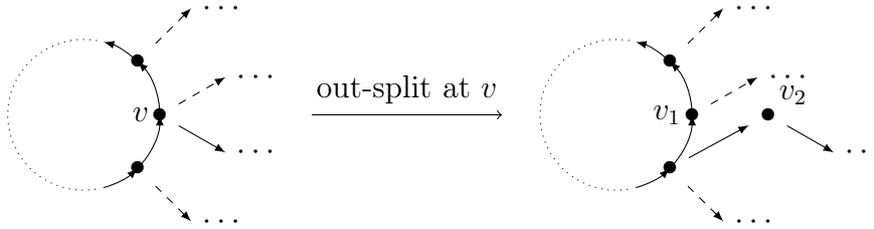

    In this manner, one can move the sources of trails along the source cycle, at the cost of lengthening the trails. 
     Of course, one can also do the opposite process (via out-amalgamations) and shorten a trail.
     
    Thus, we can arrange the trails so that every trail is sourced at the same vertex $v_E$.
   
    A similar procedure can also be done on the sink cycle (via in-splits). Thus, by in-amalgamations at the sink cycle, we can shorten the trails at the cost of moving their ranges back along the sink cycle, with the result being that all trails have only one edge. 
    
    In short: Starting with a meteor graph \(E\), we can perform in- and out-splits in its trails and obtain another meteor graph whose trails do not intersect. Then, we can perform out-splits on the vertices of the source cycle, and make all trails have the same source. Performing in amalgamations on the vertices of the sink cycle, we make all trails have length \(1\). 
\end{proof}

The following result is a consequence of the proof above.
\begin{corollary}
\label{cor:splits-preserve-meteor}
    If $E$ is a meteor graph, then in- or out-splitting $E$ will preserve the lengths of the cycles of $E$.  In particular, in- and out-splitting and -amalgamation take meteor graphs to meteor graphs.
\end{corollary}

{\begin{proof}
  In- or out-splitting at a vertex in the interior of a trail does not affect the source or range of any edge on the cycles of $E$.  Furthermore, one cannot in-split non-trivially at a vertex on the source cycle, nor out-split non-trivially at a vertex on the sink cycle.  Thus, suppose we have in-split at a vertex $v$ on the sink cycle.  Without loss of generality, assume  $\mathscr E_1$ contains the single edge $e$ from $C_1^E \cap r^{-1}(v)$.  There is a unique sink cycle in the in-split graph, resulting from replacing $e$ in $C_1^E$  with $e_1$ and $v = r(e)$ with $v_1$.  An analogous argument shows that out-splitting at a vertex $v \in C^E_0$ transforms $C^E_0$ into exactly one new cycle, which has the same length as $C^E_0$.
\end{proof}}

\begin{corollary}
\label{cor:normal-form-rep}
For any meteor graph $E$, there is a meteor graph $N_E$ in normal form with $T_E \cong T_{N_E}$.
\end{corollary}
\begin{proof}
    Since meteor graphs are finite and have no sinks, \cite[Theorems 4.4 and 4.7]{Luiz} implies that in-splits and out-splits (and hence in- and out-amalgamations) preserve the talented monoids. Thus the Corollary follows from Proposition \ref{prop:convert-to-normal}.
\end{proof}

The isomorphism of talented monoids induced by an in-splitting does not preserve the order unit.  However, we have the following result.

\begin{proposition}
\label{prop:quasi-normal-form}
    If $E$ is a meteor graph such that every vertex $u$ in the interior of a trail has $|r^{-1}(u)|=1$, then we can convert $E$ into a meteor graph $N_E$ in quasi-normal form using only out-splits and -amalgamations.  Consequently, $T_E \cong T_{N_E}$ via a graded isomorphism which preserves the order unit.
\end{proposition}
\begin{proof}
    As in the first step in the proof of Proposition \ref{prop:convert-to-normal}, we can convert $E$ into a meteor graph in which $|s^{-1}(u)|=1 $ for all vertices $u$ lying in the interior of a trail.  Since $|r^{-1}(u)| = 1$ for all such vertices, by hypothesis, we conclude that the (interiors of the) trails of $E$ are disjoint.
    
    If a vertex $v$ on the source cycle $C^E_0$ of $E$ is the source of a trail, out-splitting at $v$ (via a partition which separates the first edge of the trail from any other out-going edges) moves the source of the trail one step backward around $C^E_0$, and lengthens the trail by one edge.  Consequently, we can use out-amalgamation to shorten all of the trails of $E$ until they have length 1, at the cost of allowing their sources to lie at any vertex on $C^E_0$.
    
    As out-splitting (and hence out-amalgamation) preserve the order unit of the talented monoid by \cite[Theorem 4.7]{Luiz}, the result follows.
\end{proof}

Fix a vertex $w_E$ in the sink cycle of a meteor graph $E$. By~\cite[Lemma~5.6]{hazli},  $w_E$ is a minimal periodic element. In fact we can say more. 

\begin{definition}\label{def:minimal-elt}
    A $\Gamma$-{\em minimum}  element in  a $\Z$-monoid $M$ is a non-zero element $y \in M$ such that for all nonzero $x \in M$, there exists $m \in \Z$ such that ${}^m x \geq y$.  
     \end{definition}

Clearly, if $y$ is $\Gamma$-minimum, then ${}^ny$, for any $n\in \mathbb Z$, is also $\Gamma$-minimum. It is possible a $\Gamma$-monoid $M$ has several $\Gamma$-minimum elements.  However if the minimum elements coincide up to their orbits, then we say $M$ has a unique $\Gamma$-minimum element. More formally, We say that $y$ is the {\em unique} $\Gamma$-minimum element of $M$ if  $y$ is $\Gamma$-minimum and for any  $\Gamma$-minimum element $x$ we have  $O(x)=O(y)$.

Fix a vertex $w_E$ in the sink cycle of a meteor graph $E$.
The following lemma establishes that $w_E$ is the unique minimal element of a  meteor graph.
\begin{lemma}
\label{lem:unique-min}
    Let $E$ be a meteor graph, and fix vertices \(v_E\in C^E_0\) and \(w_E\in C^E_1\). Then 
    $w_E$ is the unique minimum element of $T_E$, as in Definition \ref{def:minimal-elt}.
\end{lemma}
\begin{proof}
 Given a nonzero element $x \in T_E$, write $ x = \sum_i c_i {}^{r_i} v_E + \sum_j d_j {}^{s_j} w_E$ with $c_i, d_j\in \N$ (see Proposition~\ref{ghostagain}).  If some $d_j$ is nonzero, then $x \geq {}^{s_j} w_E$.  If $d_j  = 0$ for all $j$, then at least one $c_i$ is nonzero.  Let $n$ be the length of a path $P$ in $E$ with $s(P)=v_E$ and $r(P)=w_E$.  Then there exists $y\in T_E$ such that ${}^{r_i} v_E(n) = {}^{r_i } w_E + y$.  That is, ${}^{r_i} v_E \geq {}^{r_i-n} w_E $ and hence $x \geq {}^{r_i - n} w_E.$  Consequently, ${}^{m} x \geq {}^0 w_E,$, where $m=n-r_i$. 
    
    To see that $w_E$ is unique, choose $y\not= {}^n w_E$.  That is, we have $y = \sum_i c_i {}^{r_i} v_E + \sum_j d_j {}^{s_j} w_E$ with either some $c_i$ nonzero, or $j \geq 2$.  If we have $j \geq 2$, then there is no $m \in \Z$ such that  ${}^m w_E  \geq y$, because $w_E$ lies only on one cycle.  If $c_i \not=0$ for some $i$, and there is $m\in \Z, x \in T_E$ such that ${}^m w_E = y + x,$  
    then we again get a contradiction, because any flow of ${}^{r_i} v_E$ will include ${}^kv_E$ for some $k$.
\end{proof}

We now introduce an equivalence relation on meteor graphs which will enable us to prove that shift equivalence of meteor graphs implies strong shift equivalence.

Let \(E\) be a meteor graph. We will write 
\(p^{E}\) and \(q^{E}\), respectively, for the lengths of the source and sink cycle of $E$. We denote
\[\p_E=\operatorname{gcd}(p^{E}, q^{E})\]
and call this number the \emph{period} of \(E\).  Corollary \ref{cor:splits-preserve-meteor} tells us that, in particular, in- and out-splits and -amalgamations preserve the period of a meteor graph.



Given vertices \(v^E\in C^E_0\) and \(w_E\in C^E_1\),  let us call the triple \((E,v^E,w^E)\) a \emph{pointed meteor graph}. Let us denote by $\mathbb{Z}/\p_E$ the ring of integers modulo \(\p_E\). Given \(n\in\mathbb{Z}/\p_E\), let \(N_E(n,v^E,w^E)\) be the number of trails \(\alpha\) of \(E\) for which
\[ |v^Es(\alpha)|+|\alpha|+|r(\alpha)w^E|=n\bmod \p_E,\label{eq.length_path_trail}\]
where \(|v^Es(\alpha)|\) is the length of the {minimal path} 
from \(v^E\) to \(s(\alpha)\), and similarly for \(|r(\alpha)w^E|\). In other words, the left-hand-side of the equation above is simply the length of the minimal path from \(v^E\) to \(w^E\) which contains \(\alpha\).

For pointed meteor graphs \((E,v^E,w^E)\) and \((F,v^F,w^F)\) with source and sink cycles of respective equal length, and in particular of the same period \( \p:= \p_E = \p_F\), we define a relation
\begin{align}
     & q^{E} = q^{F},\quad p^{E}=p^{F}\text{ and there exists}\notag
        \\
      (E,v^E,w^E)\approx (F,v^F,w^F) \quad \iff \qquad &
            t\in\mathbb{Z}/\p\text{ such that for all }n\in\mathbb{Z}/\p\text{,} 
            \label{eq:condition-on-normal-form}\\
        &
           N_E(n+t,v^E,w^E)=N_F(n,v^F,w^F)\notag
\end{align}

\begin{example}
    Consider the pointed meteor graph $E$ below:
    
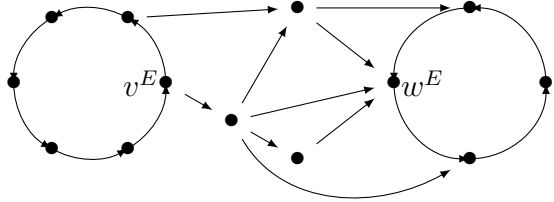
\begin{figure}[!ht]
    \centering
    \begin{tikzpicture}
        \node (center1) at (0,0) {};
        
        \node [label={[label distance=-0.3cm]180:$v^E$}] (a1) at ([shift=(0:1)]center1) {\(\bullet\)};
        \node (a2) at ([shift={(60:1)}]center1) {$\bullet$};
        \node (a3) at ([shift={(120:1)}]center1) {$\bullet$};
        \node (a4) at ([shift={(180:1)}]center1) {$\bullet$};
        \node (a5) at ([shift={(240:1)}]center1) {$\bullet$};
        \node (a6) at ([shift={(300:1)}]center1) {$\bullet$};

        \draw[-latex] (a1) arc (0:60:1);
        \draw[-latex] (a2) arc (60:120:1);
        \draw[-latex] (a3) arc (120:180:1);
        \draw[-latex] (a4) arc (180:240:1);
        \draw[-latex] (a5) arc (240:300:1);
        \draw[-latex] (a6) arc (300:360:1);
        
        \node (u) at ([shift={(-30:1)}]a1) {$\bullet$};
        \node (v1) at ([shift={(30:2)}]a1) {$\bullet$};
        \node (v2) at ([shift={(-30:2)}]a1) {$\bullet$};
        
        \node (center2) at ([shift={(0:5)}]center1) {};
        
        \node [label={[label distance=-0.3cm]0:$w^E$}] (b1) at ([shift={(180:1)}]center2) {$\bullet$};
        \node (b2) at ([shift={(270:1)}]center2) {$\bullet$};
        \node (b3) at ([shift={(0:1)}]center2) {$\bullet$};
        \node (b4) at ([shift={(90:1)}]center2) {$\bullet$};

        \draw[-latex] (b1) arc (-180:-90:1);
        \draw[-latex] (b2) arc (-90:0:1);
        \draw[-latex] (b3) arc (0:90:1);
        \draw[-latex] (b4) arc (90:180:1);
        
        \draw[-latex] (a1) to (u);
        \draw[-latex] (a2) to (v1);
        \draw[-latex] (u) to (v1);
        \draw[-latex] (u) to (v2);
        \draw[-latex] (u) to (b1);
        \draw[-latex] (v1) to (b4);
        \draw[-latex] (v1) to (b1);
        \draw[-latex] (v2) to (b1);
        \draw[-latex] (u) to [out=-60,in=210] (b2);
    \end{tikzpicture}
\end{figure}

In this case, the period of $E$ is $\p_E=\operatorname{gcd}(6,4)=2$, so $\mathbb{Z}/\!\p_E=\left\{0,1\right\}$.

To obtain a simpler meteor graph which is $\approx$-equivalent to $E$, we need to determine the numbers $N_E(0):=N_E(0,v^E,w^E)$ and $N_E(1):=N_E(1,v^E,w^E)$. For this, we need to connect $v^E$ to $w^E$ through simple paths along each possible trail $\alpha$ and check the length of this path (this will give us the left-hand side of Equation \eqref{eq.length_path_trail}), and then count how many of these lengths are even or odd (i.e., $=0$ or $1\mod p_E$, as in the right-hand side of Equation \eqref{eq.length_path_trail}).

To simplify this process, we can perform exactly the same procedure as in Figure \ref{fig.removing.intersections}, which ``separates'' the trails while preserving their sources, ranges, and lengths. (This argument is repeated in the proof of item (2) of Proposition \ref{prop.graphmovesmakeconditiongcd}). The resulting graph $E'$ is

\begin{figure}[!ht]
    \centering
    \begin{tikzpicture}
        \node (center1) at (0,0) {};
        
        \node [label={[label distance=-0.3cm]180:$v^{E'}$}] (a1) at ([shift=(0:1)]center1) {\(\bullet\)};
        \node (a2) at ([shift={(60:1)}]center1) {$\bullet$};
        \node (a3) at ([shift={(120:1)}]center1) {$\bullet$};
        \node (a4) at ([shift={(180:1)}]center1) {$\bullet$};
        \node (a5) at ([shift={(240:1)}]center1) {$\bullet$};
        \node (a6) at ([shift={(300:1)}]center1) {$\bullet$};

        \draw[-latex] (a1) arc (0:60:1);
        \draw[-latex] (a2) arc (60:120:1);
        \draw[-latex] (a3) arc (120:180:1);
        \draw[-latex] (a4) arc (180:240:1);
        \draw[-latex] (a5) arc (240:300:1);
        \draw[-latex] (a6) arc (300:360:1);
        
        \node (u) at ([shift={(-30:1)}]a1) {};
        \node (u1) at ([shift={(-90:0)}]u) {$\bullet$};
        \node (u2) at ([shift={(-90:0.3)}]u1) {$\bullet$};
        \node (u3) at ([shift={(-90:0.3)}]u2) {$\bullet$};
        \node (u4) at ([shift={(-90:0.3)}]u3) {$\bullet$};
        \node (u5) at ([shift={(-90:0.3)}]u4) {$\bullet$};
        \node (v1) at ([shift={(30:2)}]a1) {};
        \node (v11) at ([shift={(90:0.3)}]v1) {$\bullet$};
        \node (v12) at ([shift={(-90:0)}]v1) {$\bullet$};
        \node (v13) at ([shift={(-90:0.3)}]v1) {$\bullet$};
        \node (v14) at ([shift={(-90:0.3)}]v13) {$\bullet$};
        \node (v2) at ([shift={(-30:2)}]a1) {$\bullet$};
        
        \node (center2) at ([shift={(0:5)}]center1) {};
        
        \node [label={[label distance=-0.3cm]0:$w^{E'}$}] (b1) at ([shift={(180:1)}]center2) {$\bullet$};
        \node (b2) at ([shift={(270:1)}]center2) {$\bullet$};
        \node (b3) at ([shift={(0:1)}]center2) {$\bullet$};
        \node (b4) at ([shift={(90:1)}]center2) {$\bullet$};

        \draw[-latex] (b1) arc (-180:-90:1);
        \draw[-latex] (b2) arc (-90:0:1);
        \draw[-latex] (b3) arc (0:90:1);
        \draw[-latex] (b4) arc (90:180:1);
        
        \draw[-latex] (a1) to (u1);
        \draw[-latex] (a1) to (u2);
        \draw[-latex] (a1) to (u3);
        \draw[-latex] (a1) to (u4);
        \draw[-latex] (a1) to (u5);
        \draw[-latex] (a2) to (v11);
        \draw[-latex] (a2) to (v13);
        \draw[-latex] (u1) to [out=90,in=180] (v12);
        \draw[-latex] (u2) to [out=0,in=-90] (v14);
        \draw[-latex] (u3) to (b1);
        \draw[-latex] (u4) to (v2);
        \draw[-latex] (u5) to (b2);
        \draw[-latex] (v11) to (b4);
        \draw[-latex] (v12) to (b4);
        \draw[-latex] (v13) to (b1);
        \draw[-latex] (v14) to (b1);
        \draw[-latex] (v2) to (b1);
    \end{tikzpicture}
\end{figure}
and in this manner it is easier to count how many simple paths of each length connecting $v^{E'}$ and $w^{E'}$ there are:
\begin{itemize}
    \item length 2: 1
    \item length 3: 3
    \item length 4: 2
    \item length 5: 1
\end{itemize}
Therefore, $N_E(0)=3$ and $N_E(1)=4$.

Thus, to construct a simpler graph that is $\approx$-equivalent to $E$, we may start with the same sink and source cycles and simply add $3$ trails of even length and $4$ trails of odd length. For example, consider the graphs $F$ and $F'$ in Figure  6:

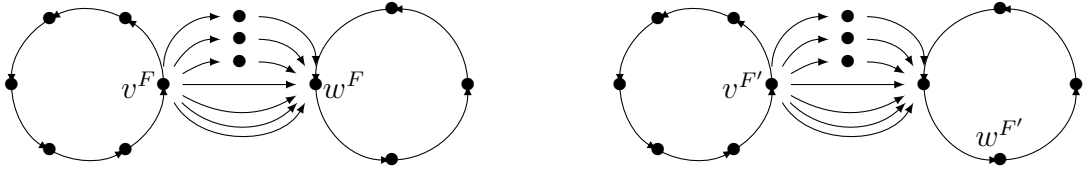
\begin{figure}[!ht]
    \centering
    \begin{tikzpicture}
        \node (center1) at (0,0) {};
        
        \node [label={[label distance=-0.3cm]180:$v^{F}$}] (a1) at ([shift=(0:1)]center1) {\(\bullet\)};
        \node (a2) at ([shift={(60:1)}]center1) {$\bullet$};
        \node (a3) at ([shift={(120:1)}]center1) {$\bullet$};
        \node (a4) at ([shift={(180:1)}]center1) {$\bullet$};
        \node (a5) at ([shift={(240:1)}]center1) {$\bullet$};
        \node (a6) at ([shift={(300:1)}]center1) {$\bullet$};

        \draw[-latex] (a1) arc (0:60:1);
        \draw[-latex] (a2) arc (60:120:1);
        \draw[-latex] (a3) arc (120:180:1);
        \draw[-latex] (a4) arc (180:240:1);
        \draw[-latex] (a5) arc (240:300:1);
        \draw[-latex] (a6) arc (300:360:1);
        
        \node (u) at ([shift={(0:2)}]center1) {};
        \node (u1) at ([shift={(90:0.3)}]u) {$\bullet$};
        \node (u2) at ([shift={(90:0.3)}]u1) {$\bullet$};
        \node (u3) at ([shift={(90:0.3)}]u2) {$\bullet$};
        
        \node (center2) at ([shift={(0:4)}]center1) {};
        
        \node [label={[label distance=-0.3cm]0:$w^{F}$}] (b1) at ([shift={(180:1)}]center2) {$\bullet$};
        \node (b2) at ([shift={(270:1)}]center2) {$\bullet$};
        \node (b3) at ([shift={(0:1)}]center2) {$\bullet$};
        \node (b4) at ([shift={(90:1)}]center2) {$\bullet$};

        \draw[-latex] (b1) arc (-180:-90:1);
        \draw[-latex] (b2) arc (-90:0:1);
        \draw[-latex] (b3) arc (0:90:1);
        \draw[-latex] (b4) arc (90:180:1);

        \draw[-latex] (a1) to [out=30,in=180] (u1);
        \draw[-latex] (a1) to [out=60,in=180] (u2);
        \draw[-latex] (a1) to [out=90,in=180] (u3);
        \draw[-latex] (u1) to [out=0,in=150] (b1);
        \draw[-latex] (u2) to [out=0,in=120] (b1);
        \draw[-latex] (u3) to [out=0,in=90] (b1);
        \draw[-latex] (a1) to (b1);
        \draw[-latex] (a1) to [out=-30,in=210] (b1);
        \draw[-latex] (a1) to [out=-45,in=225] (b1);
        \draw[-latex] (a1) to [out=-60,in=240] (b1);


        \node (center1) at (8,0) {};
        
        \node [label={[label distance=-0.3cm]180:$v^{F'}$}] (a1) at ([shift=(0:1)]center1) {\(\bullet\)};
        \node (a2) at ([shift={(60:1)}]center1) {$\bullet$};
        \node (a3) at ([shift={(120:1)}]center1) {$\bullet$};
        \node (a4) at ([shift={(180:1)}]center1) {$\bullet$};
        \node (a5) at ([shift={(240:1)}]center1) {$\bullet$};
        \node (a6) at ([shift={(300:1)}]center1) {$\bullet$};

        \draw[-latex] (a1) arc (0:60:1);
        \draw[-latex] (a2) arc (60:120:1);
        \draw[-latex] (a3) arc (120:180:1);
        \draw[-latex] (a4) arc (180:240:1);
        \draw[-latex] (a5) arc (240:300:1);
        \draw[-latex] (a6) arc (300:360:1);
        
        \node (u) at ([shift={(0:2)}]center1) {};
        \node (u1) at ([shift={(90:0.3)}]u) {$\bullet$};
        \node (u2) at ([shift={(90:0.3)}]u1) {$\bullet$};
        \node (u3) at ([shift={(90:0.3)}]u2) {$\bullet$};
        
        \node (center2) at ([shift={(0:4)}]center1) {};
        
        \node (b1) at ([shift={(180:1)}]center2) {$\bullet$};
        \node (b2) [label={[label distance=-0.2cm]90:$w^{F'}$}] at ([shift={(270:1)}]center2) {$\bullet$};
        \node (b3) at ([shift={(0:1)}]center2) {$\bullet$};
        \node (b4) at ([shift={(90:1)}]center2) {$\bullet$};

        \draw[-latex] (b1) arc (-180:-90:1);
        \draw[-latex] (b2) arc (-90:0:1);
        \draw[-latex] (b3) arc (0:90:1);
        \draw[-latex] (b4) arc (90:180:1);

        \draw[-latex] (a1) to [out=30,in=180] (u1);
        \draw[-latex] (a1) to [out=60,in=180] (u2);
        \draw[-latex] (a1) to [out=90,in=180] (u3);
        \draw[-latex] (u1) to [out=0,in=150] (b1);
        \draw[-latex] (u2) to [out=0,in=120] (b1);
        \draw[-latex] (u3) to [out=0,in=90] (b1);
        \draw[-latex] (a1) to (b1);
        \draw[-latex] (a1) to [out=-30,in=210] (b1);
        \draw[-latex] (a1) to [out=-45,in=225] (b1);
        \draw[-latex] (a1) to [out=-60,in=240] (b1);
    \end{tikzpicture}
    \caption{The graph $F$ is on the left and $F'$ on the right.}
\end{figure}

By the previous paragraph, $F\approx E$. The graph $F'$ was created by moving the selected vertex $w^F$ one step along the sink cycle, which increases the length of all minimal paths connecting the selected vertices by $1$. So we also have $F'\approx F$, by taking $t=1$ in Equation \eqref{eq:condition-on-normal-form}.
\end{example}

\begin{proposition}\label{prop.graphmovesmakeconditiongcd}
    The relation \(\approx\) as above is an equivalence relation on the class of meteor graphs of period \(\p\). Moreover:
    \begin{enumerate}[label=(\arabic*)]
        \item Any two pointed meteor graphs \((E,v^E_1,w^E_1)\) and \((E,v^E_2,w^E_2)\) with the same underlying graph \(E\) belong to the same equivalence class under \(\approx\).
        \item If \(F\) is obtained from the meteor graph \(E\) by performing in-\ and out-splits and -{a\-mal\-ga\-ma\-tions}, then any pointed graphs with underlying graphs \(E\) and \(F\) belong to the same equivalence class under \(\approx\).
    \end{enumerate}
\end{proposition}
\begin{proof}
A moment's thought reveals that $\approx$ is  transitive, symmetric, and reflexive.
    \begin{enumerate}[label=(\arabic*)]
        \item Suppose that \(v^E_2\) is the vertex right after \(v^E_1\) in \(C^E_0\) and that \(w^E_1=w^E_2\) (which we call simply \(w^E\)). To see that \((E,v^E_1,w^E)\approx (E,v^E_2,w^E)\), notice that, for a trail \(\alpha\) of \(E\):
        \begin{itemize}
            \item If \(s(\alpha)\neq v^E_1\), then the minimal path from \(v^E_2\) to \(w^E\) which passes through \(\alpha\) is the same as the one from \(v^E_1\) to \(w^E\), except the first edge.
            \item If \(s(\alpha)=v^E_1\), then to obtain the minimal path from \(v^E_2\) to \(w^E\) which passes through \(\alpha\) we have to concatenate the one from \(v^E_1\) to \(w^E\) with the path from \(v^E_2\) to \(v^E_1\) inside the cycle \(c^E_0\), and this adds \(p^E_0-1\) to its length, which \(\pmod{\p_E}\) is the same as subtracting \(1\) (as \(\p_E\) divides \(p^E\)).
        \end{itemize}
        In any case, we conclude that, for any \(n\in\mathbb{Z}/\p_E\),
        \(N_E(n,v^E_1,w^E)=N_E(n-1,v^E_2,w^E)\)
        and therefore \((E,v^E_1,w^E)\approx (E,v^E_2,w^E)\), as long as \(v^E_2\) is the point right after \(v^E_1\) in \(C_0^E\). Since \(\approx\) is  transitive and any vertex in \(C_0^E\) can be reached from any other one by moving along the cycle, we see that \((E,v^E_1,w^E)\approx (E,v^E_2,w^E)\) for any \(v^E_1,v^E_2\in C^E_0\).
        
        A similar argument allows us to change the vertex $w^E$ in \(C^E_1\) and preserve \(\approx\)-classes, which completes the proof of this item.
        
    \item Fix vertices \(v^E\) and \(w^E\) in the source and sink cycles of \(E\), respectively.
    
    If the graph \(F\) is obtained by performing an in-split at a vertex \(v\) of \(E\) which belongs to the interior of the trails, then each trail \(\alpha\) which passes through \(v\) will correspond to a single trail \(\alpha'\) in \(F\) of the same length, and the chosen points $v^E$ of \(C_0^E\) and $w^E$ of \(C_1^E\) will remain unmodified, as will all other trails:
    
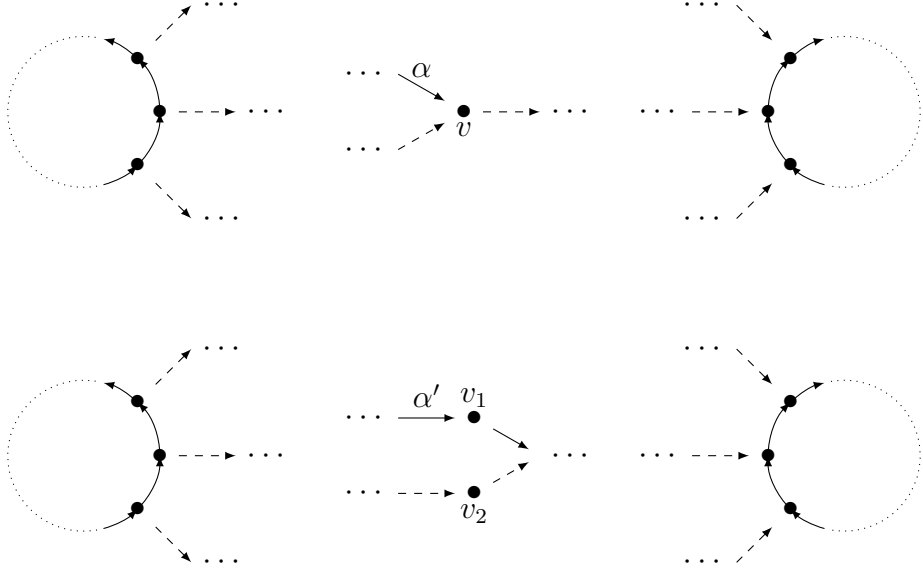
\[\begin{tikzpicture}
        \node (centerright) at (0,0) {};
        \node (uright) at ([shift=(225:1)]centerright) {\(\bullet\)};
        \node (middleright) at ([shift=(180:1)]centerright) {\(\bullet\)};
        \node (wright) at ([shift=(135:1)]centerright) {\(\bullet\)};

        \draw[-latex] (centerright) ++(255:1) arc (255:225:1);
        \draw[-latex] (centerright) ++(225:1) arc (225:180:1);
        \draw[-latex] (centerright) ++(180:1) arc (180:135:1);
        \draw[-latex] (centerright) ++(135:1) arc (135:105:1);
        \draw[-latex,dashed] (centerright) ++(225:2) node[left] {$\cdots$} -- (uright);
        \draw[-latex,dashed] ([shift=(180:1)]middleright.center) node[left] {$\cdots$} -- (middleright);
        \draw[-latex,dashed] (centerright) ++(135:2) node[left] {$\cdots$} -- (wright);
        
        \draw[dotted] (centerright) ++ (100:1) arc (100:-100:1);
        
        \node (centermiddle) at ([shift={(180:5)}]centerright) {};
        
        \node (umiddle) at (centermiddle) {$\bullet$};
        \draw[-latex] ([shift=({150:1})]umiddle.center) node[left] {$\cdots$} to node[midway, above] {$\alpha$} (umiddle) node[below] {$v$};
        \draw[-latex,dashed] ([shift=({210:1})]umiddle.center) node[left] {$\cdots$} to (umiddle);
        \draw[-latex,dashed] (umiddle) -- ([shift=(0:1)]umiddle.center) node[right] {$\cdots$};

        \node (centerleft) at ([shift=(180:5)]centermiddle) {};
        \node (bottomleft) at ([shift=(-45:1)]centerleft) {\(\bullet\)};
        \node (middleleft) at ([shift=(0:1)]centerleft) {\(\bullet\)};
        \node (topleft) at ([shift=(45:1)]centerleft) {\(\bullet\)};

        \draw[-latex] (centerleft) ++(-75:1) arc (-75:-45:1);
        \draw[-latex] (centerleft) ++(-45:1) arc (-45:0:1);
        \draw[-latex] (centerleft) ++(0:1) arc (0:45:1);
        \draw[-latex] (centerleft) ++(45:1) arc (45:75:1);
        \draw[-latex,dashed] (bottomleft) -- ([shift=(-45:1)]bottomleft.center) node[right] {$\cdots$};
        \draw[-latex,dashed] (middleleft) -- ([shift=(0:1)]middleleft.center) node[right] {$\cdots$};
        \draw[-latex,dashed] (topleft) -- ([shift=(45:1)]topleft.center) node[right] {$\cdots$};
        
        \draw[dotted] (centerleft) ++ (-80:1) arc (-80:-280:1);
        \end{tikzpicture}\]
        
        \bigskip
        
        \[\begin{tikzpicture}
        \node (centerright) at (0,0) {};
        \node (uright) at ([shift=(225:1)]centerright) {\(\bullet\)};
        \node (middleright) at ([shift=(180:1)]centerright) {\(\bullet\)};
        \node (wright) at ([shift=(135:1)]centerright) {\(\bullet\)};

        \draw[-latex] (centerright) ++(255:1) arc (255:225:1);
        \draw[-latex] (centerright) ++(225:1) arc (225:180:1);
        \draw[-latex] (centerright) ++(180:1) arc (180:135:1);
        \draw[-latex] (centerright) ++(135:1) arc (135:105:1);
        \draw[-latex,dashed] (centerright) ++(225:2) node[left] {$\cdots$} -- (uright);
        \draw[-latex,dashed] ([shift=(180:1)]middleright.center) node[left] {$\cdots$} -- (middleright);
        \draw[-latex,dashed] (centerright) ++(135:2) node[left] {$\cdots$} -- (wright);
        
        \draw[dotted] (centerright) ++ (100:1) arc (100:-100:1);
        
        \node (centermiddle) at ([shift={(180:5)}]centerright) {};
        
        \node (A) at ([shift=(0:1)]centermiddle) {};
        \node (v1) at ([shift=(150:1)]A) {$\bullet$};
        \node (v2) at ([shift=(210:1)]A) {$\bullet$};
        \draw[-latex] ([shift=({180:1})]v1.center) node[left] {$\cdots$} to node[midway, above] {$\alpha'$} (v1) node[above] {$v_1$};
        \draw[-latex] (v1) to (A);
        \draw[-latex,dashed] ([shift=({180:1})]v2.center) node[left] {$\cdots$} to (v2) node[below] {$v_2$};
        \draw[-latex,dashed] (v2) -- (A) node[right] {$\cdots$};

        \node (centerleft) at ([shift=(180:5)]centermiddle) {};
        \node (bottomleft) at ([shift=(-45:1)]centerleft) {\(\bullet\)};
        \node (middleleft) at ([shift=(0:1)]centerleft) {\(\bullet\)};
        \node (topleft) at ([shift=(45:1)]centerleft) {\(\bullet\)};

        \draw[-latex] (centerleft) ++(-75:1) arc (-75:-45:1);
        \draw[-latex] (centerleft) ++(-45:1) arc (-45:0:1);
        \draw[-latex] (centerleft) ++(0:1) arc (0:45:1);
        \draw[-latex] (centerleft) ++(45:1) arc (45:75:1);
        \draw[-latex,dashed] (bottomleft) -- ([shift=(-45:1)]bottomleft.center) node[right] {$\cdots$};
        \draw[-latex,dashed] (middleleft) -- ([shift=(0:1)]middleleft.center) node[right] {$\cdots$};
        \draw[-latex,dashed] (topleft) -- ([shift=(45:1)]topleft.center) node[right] {$\cdots$};
        
        \draw[dotted] (centerleft) ++ (-80:1) arc (-80:-280:1);
        \end{tikzpicture}\]
        In this manner,
        \[N_E(n,v^E,w^E)=N_F(n,v^E,w^E).\]
        
        Similarly, out-splitting at vertices in the interior of the trails will preserve the number and length of the trails, so performing out-splits at these vertices 
        also preserves \(\approx\)-equivalence classes.
        
        If we perform an in-split at a vertex \(v\) in the sink cycle, then any trail \(\alpha\) which passes through \(v\) is transformed into a trail \(\alpha'\) of length \(|\alpha|+1\), with its range moved one step along the sink cycle, while all other trails remain the same (see Figure \ref{fig.insplit.sink.loop}). Consequently,
        \[|\alpha|+|v^Es(\alpha)|+|r(\alpha)w^E|=|\alpha'|+|v^Es(\alpha')|+|r(\alpha')w^E|\mod q^E\]
        (by the same argument as in item (1), considering the cases where \(r(\alpha)\) equals \(w^E\) or not), so the equality also holds \(\mod \p_E\). Again, we obtain, for all \(n\in\mathbb{Z}/ \p_E\),
        \(N_E(n,v^E,w^E)=N_F(n,v^E,w^E).\)
        
        Out-splits at vertices in the source cycle are dealt with similarly. In-splits at the source cycles are not permitted,  nor are out-splits at the sink cycle,  as all vertices in the source cycle have $|r^{-1}(v)|=1$, and vertices in the sink cycle have $|s^{-1}(v)|=1$. Moreover, item (1) says that we do not need to preserve the selected points of the cycles to preserve \(\approx\)-classes.
        
        Thus, in- and out-splits preserve \(\approx\)-classes. As amalgamations are the reverse procedures to splits, they also preserve \(\approx\)-classes. Transitivity of \(\approx\) yields the desired result.\qedhere
    \end{enumerate}
\end{proof}

Therefore, the \(\approx\)-class of any pointed meteor graph does not depend on the pointings: It depends solely on the underlying graph itself. We thus have the relation
\begin{align*}
    E\approx F\iff
        &\text{for some (equivalently, any) choices of }v^E\in C_0^E\text{, }w^E\in C_1^E\text{, }v^F\in C_0^F\text{, }w^F\in C_1^F\text{,}\\
        &\text{we have }(E,v^E,w^E)\approx(F,v^F,w^F).
\end{align*}

Our goal is to prove that the ``geometric'' equivalence relation \(\approx\) classifies Leavitt path algebras of meteor graphs up to graded Morita equivalence. 


\begin{proposition}
\label{prop.conditiongcdimpliestransform}
Let $E$ and $F$ be meteor graphs. 
    If $E \approx F$, then there is a sequence of in- and out-splits and -amalgamations which transforms \(E\) into \(F\).
\end{proposition}
\begin{proof}
We begin by noting that, in a meteor graph with $p$ edges in the source cycle and $q$ edges in the sink cycle, in- and out-splits and -amalgamations can be used to lengthen or shorten any trail by  $\p_E:= \gcd(p, q)$ edges.    By B\'ezout's Theorem, there are positive integers \(\widetilde{p}\), \(\widetilde{q}\) and \(N\) such that \(\widetilde{p}p+\widetilde{q}q=
\p_E+Np\). The move of lengthening a trail by  \(\p_E
\) edges can thus be attained using in- and out-splits and amalgamations in the following manner:
     \begin{enumerate}[label=(\roman*)]
      \item Lengthen the trail by \(p\) edges \(\widetilde{p}\) times, by out-splitting at the source cycle.
       \item Lengthen the trail by \(q\) edges \(\widetilde{q}\) times, by in-splitting at the sink cycle.
    \item Shorten the trail by  \(p\) edges \(N\) times, by out-amalgamation.
    \end{enumerate}
    It follows that we can also shorten any path from $v_E$ to $w_E$ by  $\p_E$ edges, by using amalgamations rather than splittings (and vice versa) in the argument above. 

    If $E \approx F$, then we have $\p_E = \p_F =: \p$. Furthermore, by using the above procedure to repeatedly lengthen or shorten from the trails in $E$ and $F$ by $\p$ edges, we may assume the existence of $t \in \Z/\p$ such that, for all $n \in \Z$, there are precisely as many minimal-length paths of length $n+t$ from $v^E$ to $w^E$, as there are paths of length $n$ from $v^F$ to $w^F$.

   Furthermore, if $E \approx F$, then by Propositions \ref{prop:convert-to-normal} and \ref{prop.graphmovesmakeconditiongcd}, we may assume $E, F$ are in normal form, and that all trails of $E$ (respectively $F$) are sourced at $v_E$ (respectively $v_F$). Number the vertices in the sink cycle moving backwards from $w_E =: w_0^E;$ suppose there are $k_n^E$ trails with range $w_n^E$ (and similarly, there are $k_n^F$ trails with range $w_n^F$).  That is, for all $n$, there are $k_n^E$ paths from $v^E$ to $w^E$ of length $n+1$, and similarly in $F$.  Consequently, with the choice of $t$ indicated in the previous paragraph, we have 
\[ k_{n+t}^E = k_n^F \quad \text{ for all } n \in \Z/\p.\]
 Thus, the map which sends $v_E$ to $v_F$ (and, for all $n$, sends $v_n^E$ to $v_n^F$) and sends $w_n^E$ to $w_{n-t}^F$ induces an isomorphism between (the normal forms of) $E$ and $F$.
\end{proof}

We are now in a position to state our main theorem. 

\begin{theorem}\label{alergy1}
    Let \(E\) and \(F\) be essential graphs, where $E$ is a meteor graph. Then the following are equivalent.
    \begin{enumerate}[label=(\arabic*)]
        \item The Leavitt path algebras \(L(E)\) and \(L(F)\) are graded Morita equivalent,
        \item There is an order-preserving 
        $\mathbb Z[x,x^{-1}]$-module isomorphism
$K_0^{\gr}(L(E))\rightarrow K_0^{\gr}(L(F))$,        
        \item The talented monoids \(T_E\) and \(T_F\) are \(\mathbb{Z}\)-isomorphic,

        \item The graphs \(E\) and \(F\) are shift equivalent,
        
        \item The graphs \(E\) and \(F\) are strongly shift equivalent.
        
    \end{enumerate}
    \label{thm:main}
\end{theorem}
\begin{proof}

{\color{white},}

\begin{enumerate}

    \item[(1)\(\Rightarrow\)(2)] Follows from the graded Morita theory: the graded Morita equivalence gives rise to invertible bimodules, which in turn induce an isomorphism on the level of graded $K_0$ 
    (see~\cite[Theorem 2.3.7]{hazbk}).

   \item[(2)\(\Leftrightarrow\)(3)]
   
   By~\cite{hazli} the positive cone of the graded Grothendieck group $K_0^{\gr}(L(E))$ is $\mathcal V^{\gr}(L(E))$ and 
  $ \mathcal V^{\gr}(L(E)) \cong T_E$. This implies the equivalence.

   \item[(3)\(\Leftrightarrow\)(4)] This follows from Corollary~\ref{h99} and (2)\(\Leftrightarrow\)(3).

\item[(4)\(\Rightarrow\)(5)] If $E$ and $F$ are shift equivalent then, by Proposition~\ref{goldenprop}, $F$ is also a meteor graph and, by the equivalence of (3) and (4), their talented monoids are $\mathbb Z$-isomorphic. 
    Suppose that $\varphi: T_E \to T_F$ is a $\Z$-isomorphism. Denote by $p_E$ and $q_E$ the length of the source and sink cycles of $E$, respectively, and similarly $p_F$ and $q_F$ for $F$.  
    Thanks to Corollary~\ref{cor:normal-form-rep}, we may assume that $E$ and $F$ are meteor graphs in normal form (with only edges connecting the two cycles).  Indeed, we will assume that all trails in $E$ have source $v_E$, and all trails in $F$ have source $v_F$. 
    Then, Lemma~\ref{lem:unique-min} implies that $\varphi({}^0 w_E) = {}^r w_F$ for some $0\leq r \leq q_F-1$.
    Since $q_E \in \N$ is the smallest positive integer such that ${}^{q_E} w_E = w_E$, and similarly $q_F$ is the smallest positive integer satisfying ${}^{q_F} w_F  = w_F,$ it follows that $q_E = q_F =:q.$

     Now, consider the quotient of $T_F$ by the order ideal $\langle w_F \rangle = \{ {}^i w_F: 0 \leq i \leq q-1\}$. Observe that, for any vertex $v$ on the source cycle, $v \not \in \langle w_F\rangle$, and so  $T_F /\langle w_F\rangle  \cong T_{C_0^F} 
     \cong \bigoplus_{k=0}^{p_F-1}\mathbb{N}.$ 
     Similarly, \(T_E/\!\langle w_E\rangle\) is the $\Z$-cyclic monoid of rank \(p_E\). Since these two monoids are isomorphic, 
     \(p_E=p_F=: p\).

Next, we show that  
\begin{equation}\varphi(v_E) = {}^{i } v_F + \sum_{j=1}^q \ell_j ( {}^j w_F), \label{eq:phi-vE-def}
 \end{equation} 
 for some $i \in \Z$ and $\ell_j \in \N \cup \{ 0\}
    $.  

Suppose that
$\varphi(v_E) =  \sum_{i=0}^n c_i v_F(r_i) + \sum_{i=1}^m d_i w_F(s_i).$
Since $\varphi$ induces an isomorphism  $\overline \varphi: T_E/\langle w_E \rangle  \to T_F/\langle w_F \rangle$, passing to the quotient we have that
$\overline{\varphi}(v_E) =  \sum_{i=0}^n c_i v_F(r_i) .$ Since $v_E$ is a minimal element in $T_E/\langle w_E \rangle$, it follows that $n=0$ and $c_0= 1$. This gives the identity (\ref{eq:phi-vE-def}).        

Notice that the map $T_F \to T_F$ given by $x \mapsto {}^{-i} x$, for a fixed $i \in \Z$, is a $\Z$-isomorphism.  Composing with this map, we may assume that $\varphi(v_E) =  v_F +\sum_{j=1}^q \ell_j ( {}^j w_F). $

Recall that  $T_E = M_{\overline E}$ is the inductive limit of  
   \[ \cdots \longrightarrow \N^{p+q} \stackrel{A_E}{\longrightarrow} \N^{p+q} \stackrel{A_E}{\longrightarrow} \N^{p+q} \longrightarrow \cdots, \]
   where $A_E$ denotes the adjacency matrix of $E$.
Furthermore, since $A_E = \begin{pmatrix}
   C_p & * \\ 0 & C_q
   \end{pmatrix}$, where $$C_k = \begin{pmatrix}
   0 & 1 & & \cdots & \\
   & 0 & 1 & & \\
   & & \ddots & \vdots \\
   1 & 0 & \cdots & &
   \end{pmatrix}$$ denotes the adjacency matrix of the cycle of length $k$, the matrix $A_E$ is invertible (as the source and sink cycles have lengths at least $1$).  Along with the Confluence Lemma, this implies that the canonical inclusion maps $\iota_E^n: \N^{p+q} \to T_E$ associated to the inductive limit are all injective. 
  As the copies of $\N$ in $\N^{p+q}$ represent the vertices of $E$, these inclusions are related to the $\mathbb{Z}$-action via the formula
   \begin{equation}\label{eq.inclusions_and_action}\alpha_1\circ\iota_E^n = \iota_E^{n+1},\end{equation}
   where $\alpha_1\colon T_E\to T_E$ denotes the action of $1\in \Z$. Similar facts hold for the graph $F$.
   
   Choosing $n$ large enough so that $\text{im}\, \iota^n_F \supseteq\text{im}\, \varphi \circ \iota^0_E$, we obtain an injective map 
   \[ C := (\iota^n_F)^{-1} \circ \varphi \circ (\iota^0_E) : \N^{p+q} \to \N^{p+q}.\] 
   (In the equation above, we implicitly restrict the codomains of $\varphi$ and $\iota_F^n$ to the image of $\iota_F^n$.)
    Let us prove that $CA_E=A_F C$. On one hand,
    \begin{align*}
       \iota_F^{n+1}\circ C\circ A_E
            &=\iota_F^{n+1}\circ (\iota^n_F)^{-1} \circ \varphi \circ (\iota^0_E) \circ A_E\\
            &=\alpha_1\iota_F^n\circ (\iota^n_F)^{-1} \circ \varphi \circ (\iota^0_E) \circ A_E\tag{Equation \eqref{eq.inclusions_and_action}}\\
            &=\alpha_1\circ \varphi \circ (\iota^0_E) \circ A_E\\
            &=\varphi\circ\alpha_1\circ (\iota^0_E) \circ A_E\tag{$\varphi$ a $\mathbb{Z}$-isomorphism}\\
            &=\varphi\circ(\iota^1_E) \circ A_E\tag{Equation \eqref{eq.inclusions_and_action}}\\
            &=\varphi\circ(\iota^0_E).\tag{$T_E$ is a direct limit}
    \end{align*}
    On the other hand,
    \begin{align*}
        \iota_F^{n+1}\circ A_F\circ C
            &=\iota_F^{n+1}\circ A_F\circ (\iota_F^n)^{-1}\circ\varphi\circ\iota_E^0\\
            &=\iota_F^n\circ (\iota_F^n)^{-1}\circ\varphi\circ\iota_E^0.\tag{$T_F$ as a direct limit}\\
            &=\varphi\circ\iota_E^0
    \end{align*}
    This shows that $\iota_F^{n+1}\circ C\circ A_E=\varphi\circ\iota_E^0=\iota_F^{n+1}\circ A_F\circ C$. As $\iota_F^{n+1}$ is injective, we conclude that 
\begin{equation}
\label{eq:C-intertwines} C A_E = A_F C.\end{equation}

Now, as $C$ is induced from $\varphi$, the facts above about $\varphi(v_E)$ and $\varphi(w_E)$ are reflected in $C$.  That is, 
\[ C = \begin{pmatrix}
I & B \\ 0 & \tilde{C}_r
\end{pmatrix},\]
where \(\displaystyle \tilde C_r = (C_q)^r= \sum_{i=1}^q E_{i, r+i\bmod q}\) represents the fact that $\varphi({}^0 w_E) = {}^r w_F$.
 
 Using these two facts, we now analyze the structure of $B$.  For each $1 \leq j \leq q$, write $k^E_j$ for the number of edges with source $v_E$ and range (the vertex on the sink-cycle equivalent to) ${}^j w_E$, and similarly for $F$.  
 Then, the upper right block $*$ in $A_E$ has at most one nonzero row (the first row): $(k^E_q \,  k^E_1 \, k^E_2  \cdots k^E_{q-1}).$
 
 Notice that
 Equation \eqref{eq:C-intertwines} implies that $C_p B + *_F \tilde C_r = *_E + B C_q,$ and consequently 
 \begin{align*} 
 &B_{2, j} + k^F_{j-r-1} = k^E_{j-1} + B_{1, j-1} & \text{ for } 1 \leq j \leq q; \\
& B_{i+1, j} = B_{i, j-1} & \text{ for } 2 \leq i \leq p-1, 1 \leq j \leq q; \\
& B_{1, j} = B_{p, j-1} & \text{ for } 1 \leq j \leq q.
 \end{align*}
 We conclude that for any $1\leq j \leq q$, we have 
 \begin{align*}
      B_{p, j-1} = B_{1,j} = B_{2, j+1} + k^F_{j-r}  - k^E_{j}  
      &= B_{3, j+2} +  k^F_{j-r}  - k^E_{j} = \cdots \\  \cdots &= B_{p, j+p-1} + k^F_{j-r} - k^E_{j},
    \end{align*}
 where the addition on the indices is performed modulo $q$.
 
 Applying the same logic to $B_{p, j+p-1},$ we obtain that 
 \[ B_{p, j-1} = B_{p, j +2p-1} + k^F_{j-r} - k^E_{j}+k^F_{j+p-r} -k^E_{j+p}.\]
 
 Now, write $g = \gcd(p,q)$ and write $p = cg, q = dg$, where $\gcd(c, d) = 1$.  Then 
 \[B_{p, j-1} = B_{p, j + cq -1} = B_{p, j+ dp -1}.\]  It  follows from the computations above that we must have 
 \begin{equation}
 \sum_{n=0}^{d-1} k^F_{np +j -r} = \sum_{n=0}^{d-1} k^E_{j+np}.
 \label{eq:B-forces-edges}
 \end{equation}
 Observe that 
 for any $n$ we have $n p = n cg \equiv 0 \bmod g$.  Since the addition on the indices is performed modulo $q$,  w
 e conclude that the right-hand sum is precisely the number of paths in $E$ of length $j \pmod g$, and the left-hand sum is precisely the number of paths in $F$ of length $j-r \pmod g$.
 
  Consequently, for any meteor graphs for which $T_E \cong_\Z T_F$, Condition \eqref{eq:condition-on-normal-form} must hold (with $t = -r$).  By Proposition~\ref{prop.conditiongcdimpliestransform}, $E$ can be transformed into $F$ via in- and out-splittings and amalgamations. Now, Theorem~\ref{willimove}
    implies that the graphs $E$ and $F$ are strongly shift equivalent.

  \item[(5)\(\Rightarrow\)(1)] Since $E$ is strongly shift equivalent to $F$, there is a sequence of out-splittings, in-splittings, and the inverses of these, namely, out-amalgamations, and in-amalgamation (Theorem~\ref{willimove}) that transforms $E$ into $F$. By~\cite[Proposition 15]{hazd} it follows that $L(E)$ is graded Morita equivalent to $L(F)$.\qedhere
 \end{enumerate}
\end{proof}

Putting together our result with several established results in the literature, we can now relate the notions of Morita theory in algebras and operator algebras for the class of meteor graphs. This result is related to the questions on the relationship between the Morita theory of Leavitt path algebras and graph $C^*$-algebras (see~\cite[\S5.6]{lpabook} and \cite[Theorem~14.7]{eilers}). 

For our final result we use the equivariant $K$-theory, $K_0^{\mathbb T}$, for $C^*$-algebras \cite{chrisp}. For a finite graph $E$, 
there are canonical order isomorphisms of $\Z[x,x^{-1}]$-modules
\begin{equation}\label{noalergy}
K^{\gr}_0(L(E)) \cong K_0(L(E \times \mathbb Z)) \cong  K_0(C^*(E \times \mathbb Z))\cong K_0^{\mathbb T}(C^*(E)),
\end{equation}
where $E\times \mathbb Z$ is the covering graph of $E$. The above isomorphism has been established in several places in the literature (see for example~\cite[p. 275]{mathann}, \cite[Proof of Theorem A]{eilers2}).

\begin{corollary}\label{ruiz}
    Let $E$ and $F$ be essential graphs, where $E$ is a meteor graph. Then the Leavitt path algebras $L(E)$ and $L(F)$ are graded Morita equivalent if, and only if, the graph $C^*$-algebras $C^*(E)$ and $C^*(F)$ are equivariant Morita equivalent. 
\end{corollary}
\begin{proof}
Suppose that the Leavitt path algebras $L(E)$ and $L(F)$ are graded Morita equivalent. By Theorem~\ref{alergy1}, $F$ is also a meteor graph and $E$ and $F$ are strongly shift equivalent. Then, from \cite[Theorem 3.2]{pask}, the in- and out-splittings we used to convert from $E$ to $F$ give a Morita equivalence between $C^*(E)$ and $C^*(F)$. Furthermore, \cite[Theorem 3.4]{ef} guarantees that this equivalence is equivariant. 

For the converse, suppose that the graph $C^*$-algebras $C^*(E)$ and $C^*(F)$ are equivariantly Morita equivalent. Then $K_0^{\mathbb T}(C^*(E)) \cong K_0^{\mathbb T}(C^*(F))$ as ordered $\mathbb Z[x,x^{-1}]$-modules (see \cite[p.~297]{comb} and \cite[Prop.~2.9.1]{chrisp}). From~(\ref{noalergy}) it follows that 
$K^{\gr}_0(L(E)) \cong K^{\gr}_0(L(F))$. Again, Theorem~\ref{alergy1} gives that the Leavitt path algebras $L(E)$ and $L(F)$ are graded Morita equivalent, as desired. 
\end{proof}

{\bf Acknowledgments.}
We thank Efren Ruiz for providing us with references for Corollary~\ref{ruiz}. 
This work resulted from a Research in Pairs stay at the Mathematisches Forschungsinstitut Oberwolfach.  E. Gillaspy was partially supported by the National Science Foundation grant DMS-1800749. D. Gon\c{c}alves was partially supported by  Fundação de Amparo \`a Pesquisa e Inova\c{c}\~ao do Estado de Santa
Catarina – FAPESC, Conselho Nacional de Desenvolvimento Cient\'ifico e Tecnol\'ogico (CNPq) and Capes-PrInt grant number 88881.310538/2018-01 - Brazil. R. Hazrat acknowledges Australian Research Council grant DP230103184.

\end{document}